\documentclass[10 pt]{amsart}

\usepackage{amsmath, amssymb, amscd, cases}
\usepackage[square]{natbib}
\usepackage{graphicx}
\usepackage[dvipsnames]{xcolor}
\usepackage{caption}
\usepackage{tcolorbox}
\usepackage[margin=2.5cm]{geometry}

\usepackage{hyperref}
\hypersetup{
    colorlinks,
    citecolor=blue,
    filecolor=black,
    linkcolor=red,
    urlcolor=black
}
\bibpunct{[}{]}{,}{n}{}{;} 

\numberwithin{equation}{section}

\newtheorem{theorem}{Theorem}
\newtheorem{corollary}[theorem]{Corollary}

\newtheorem{proposition}[theorem]{Proposition}
\newtheorem{lemma}[theorem]{Lemma}

\theoremstyle{definition}
\newtheorem{definition}[theorem]{Definition}

\theoremstyle{remark}
\newtheorem{remark}[theorem]{Remark}

\numberwithin{theorem}{section}

\newcommand{\zz}{\mathbb{Z}}
\newcommand{\cc}{\mathbb{C}}
\newcommand{\qq}{\mathbb{Q}}
\newcommand{\ff}{\mathbb{F}}
\newcommand{\pp}{\mathbb{P}}
\newcommand{\gp}{\mathcal{P}}
\newcommand{\rr}{\mathbb{R}}
\newcommand{\cX}{\mathcal{X}}
\newcommand{\cO}{\mathcal{O}}
\newcommand{\cA}{\mathcal{A}}
\newcommand{\fin}{\text{fin}}
\newcommand{\tcX}{\widetilde{\cX}}
\newcommand{\Gcan}{\mathfrak{g}_{\mathrm{can}}}
\newcommand{\M}{\mathcal{M}}
\newcommand{\GaP}{\Gamma_0(p^2)}
\newcommand{\Rpar}{R^{\mathrm{par}}}
\newcommand{\Rell}{R^{\mathrm{ell}}}
\newcommand{\Rdis}{R^{\mathrm{dis}}}
\newcommand{\Rhyp}{R^{\mathrm{hyp}}}
\newcommand{\Ehyp}{\mathcal{E}^{\mathrm{hyp}}}
\newcommand{\Epar}{\mathcal{E}^{\mathrm{par}}}
\newcommand{\hh}{\mathbb{H}}
\newcommand{\GaInf}{\mathcal{B}}
\renewcommand{\L}{\mathcal{L}}
\newcommand{\smmat}[4]{\left(\begin{smallmatrix}
   #1 & #2\\
  #3 & #4\\
\end{smallmatrix}\right)}
\newcommand{\gem}{\ga_{\mathrm{EM}}}

\DeclareMathOperator{\pic}{Pic}
\DeclareMathOperator{\Div}{div}
\DeclareMathOperator{\spec}{Spec}

\newcommand{\tH}{\mathbb{H}}
\newcommand{\Z}{\mathbb{Z}}
\newcommand{\N}{\mathbb{N}}

\newcommand\SL{{\mathrm {SL}}}
\newcommand{\Ga}{\Gamma}
\newcommand{\ga}{\gamma}
\newcommand{\de}{\delta}
\newcommand{\De}{\Delta}
\newcommand{\la}{\lambda}
\newcommand{\La}{\Lambda}
\newcommand{\mb}{\mathbb}
\newcommand{\ov}{\overline}
\newcommand{\pa}{\partial}
\newcommand{\al}{\alpha}
\newcommand{\ti}{\tilde}
\newcommand{\tr}{\operatorname{tr}}
\newcommand{\muhyp}{\mu_{\mathrm{hyp}}}
\newcommand{\mucan}{\mu_{\mathrm{can}}}
\newcommand{\dis}{\operatorname{dis}}
\renewcommand{\Im}{\operatorname{Im}}
\renewcommand{\Re}{\operatorname{Re}}
\newcommand{\RG}{\mathcal{R}_{\infty}^{\Ga_0(p^2)}}
\newcommand{\CG}{\mathcal{C}_{\infty,0}^{\Gamma_0(p^2)}}
\newcommand{\C}{\mathcal{C}}
\newcommand{\Res}{\operatorname{Res}}
\newcommand{\ep}{\epsilon}
\newcommand\sE{{\mathcal{E}}}

\title{Arakelov Self-intersection numbers of minimal regular models of  
       modular curves $X_0(p^2)$}

\author{Debargha Banerjee}  
\email{debargha@iiserpune.ac.in}

\author{Diganta Borah}
\email{dborah@iiserpune.ac.in}

\author{Chitrabhanu Chaudhuri}
\email{chitrabhanu@iiserpune.ac.in}

\thanks{The first named author was partially supported by the SERB grant YSS/2015/001491. 
        The second named author was partially supported by the DST-INSPIRE grant IFA-13 MA-21}

\begin{document}

\begin{abstract}
  We compute an asymptotic expression for the  Arakelov self-intersection number of the relative
  dualizing sheaf of Edixhoven's  minimal regular model  for the modular  curve $X_0(p^2)$  over 
  $\qq$. The computation of the self-intersection numbers are used  to prove an effective version 
  of the Bogolomov conjecture for the semi-stable models of modular curves $X_0(p^2)$ and obtain 
  a bound on the stable Faltings height for those curves in a companion article~\cite{debarghachitra}.  
\end{abstract}

\maketitle

\setcounter{tocdepth}{1}
\tableofcontents{}

\section{Introduction} \label{Introduction}
In this article, we derive an asymptotic expression for the Arakelov self-intersection number of the relative dualizing sheaf of the minimal regular model over $\zz$ for the modular curve $X_0(p^2)$ in terms of its genus $g_{p^2}$ for a prime $p$. For odd, square-free $N \in \N$, this quantity was 
computed for the congruence subgroups $\Gamma_0(N)$ \cite{MR1437298}, $\Gamma_1(N)$ \cite{MR3232770} 
and recently for the principal congruence subgroups $\Gamma(N)$ \cite{Grados:Thesis}.  
 We generalize our work to semi-stable models of these modular curves in a companion article 
\cite{debarghachitra}. 
 We use the computation regarding the infinite part of the Arakelov self-intersection for the modular curve $X_0(p^2)$ done in the present paper to compute the Arakelov self-intersection numbers for semi-stable models  in \cite{debarghachitra}. 
From the viewpoint 
of Arakelov theory, the main motivation for studying the Arakelov self-intersection numbers is to prove an 
effective Bogomolov conjecture for the particular modular curve $X_0(p^2)$. Bogomolov conjecture 
was proved by Ullmo \cite{MR1609514} using Ergodic theory, though the proof is not effective.  
In \cite{debarghachitra}, we prove an effective Bogomolov conjecture and find an asymptotic expression 
of the stable Faltings heights for the modular curves of the form $X_0(p^2)$. 
In another ambitious direction, we hope that our results will 
find applications in finding Fourier coefficients of modular forms and residual Galois representations associated to modular forms following the strategy 
outlined in \cite{MR2857099}. 

The main technical difficulty of this paper lies in the fact that for square free $N$ the special fibers of 
the modular curves are reduced and even semi-stable over $\qq$, while without this hypothesis 
the special fiber is non-reduced and {\it not} semi-stable. 
We manage to remove the square-free assumption in our paper because of a careful analysis of the 
regular but non-minimal models of the corresponding modular curves, following Edixhoven  
\cite{MR1056773}.  
The bound on the self-intersection number for the infinite place in terms of Green's function has been achieved by the idea outlined 
by Zagier \cite{MR633667} using the Selberg trace formula.

Our first result concerns an asymptotic expression of the constant term of the Rankin-Selberg 
transform at the cusp $\infty$ of the Arakelov metric. To state it we need some notation. 
Let $\tH$ be the complex upper half plane and  denote the non-compact modular curve 
corresponding to the subgroup $\Ga_0(p^2)$ by $Y_0(p^2):=\Ga_0(p^2) \backslash \tH$. Let $\muhyp$ be the {\it hyperbolic measure} on the Riemann surface $X_0(p^2)$  and $v_{\GaP}$ be the volume of the compactified modular curve $X_0(p^2)$ \cite[p. 182]{MR2112196}.
We denote the weight zero Eisenstein series at the cusp $\infty$ by $E_{\infty,0}(z,s)$ and
$F$ be  the Arakelov metric on $X_0(p^2)$ (see \S~\ref{Essterm}). The Rankin-Selberg transform at the cusp $\infty$ of the Arakelov metric on $X_0(p^2)$ is defined to be 
\[
R_F(s):=\int_{Y_0(p^2)} E_{\infty,0}(z,s) F(z) \muhyp. 
\]
The above  function has a meromorphic continuation in the whole complex plane with simple pole at $s=1$ with residue $v_{\GaP}^{-1}$. Let the Laurent series expansion of the Rankin-Selberg transform at $s=1$ be given by
\[
R_F(s)=\frac{1}{v_{\GaP}(s-1)}+\mathcal{R}_{\infty}^{\Gamma_0(p^2)}+O(s-1).
\]

\begin{theorem} \label{RSconst}
  The constant term $ \mathcal{R}_{\infty}^{\Gamma_0(p^2)}$  in the Laurent series expansion of  
  the Rankin-Selberg transform of the Arakelov metric on  the modular curve $X_0(p^2)$ is asymptotically given by
  \[
   \mathcal{R}_{\infty}^{\Gamma_0(p^2)}=o\left(\frac{\log(p^2)}{g_{p^2}}\right). 
  \]
  \end{theorem}
The underlying philosophy in the proof of Theorem~\ref{RSconst} is the same as that of  Abbes-Ullmo  \cite[Theorem F, p. 6]{MR1437298},  Mayer  \cite{MR3232770} and M. Grados Fukuda~\cite{Grados:Thesis}. By invoking the Selberg trace formula, computation of the Rankin-Selberg transform of the Arakelov metric reduces to finding contribution of various motions of $\Ga_0(p^2)$ in the trace formula.
 The proof 
hinges on the simplification of suitable Eisenstein series (cf. Proposition~\ref{Es1}) that in turn is a generalization of \cite[Proposition 3. 2. 2]{MR1437298}. The actual computation of the hyperbolic [\S~\ref{hyperbolic}] and parabolic [\S~\ref{parabolic}] contributions is slightly different from the above mentioned papers because of the condition we imposed on $N$. 

Note that the elliptic contribution in the Selberg trace formula will not follow directly in the same way as that of Abbes-Ullmo \cite{MR1437298}.  In loc. cit., the authors used the square free assumption in a crucial way to factorize the Epstein zeta functions suitably  (cf.~\cite[ Lemma 3.2.4]{MR1437298}).  As Mayer  \cite{MR3232770} and M. Grados Fukuda~\cite{Grados:Thesis} worked with modular curves of the form $X_1(N)$ and $X(N)$ respectively, there is no elliptic contribution in those cases. We use an observation of  M. Grados Fukuda in conjunction to the book of 
Zagier~\cite{MR631688} to find a suitable bound on the elliptic contribution in Proposition~ \ref{contb}.  In the present paper, we give bounds on the  terms in the Laurent series expansion of a certain Zeta function that appears in the elliptic contribution [\S~\ref{elliptic}] rather than 
finding the actual expression as accomplished by Abbes-Ullmo  \cite{MR1437298}. The fact that our $N$ has only one prime factor is an advantage for us.

The computations of these three different types of contributions may be a bit complicated for general modular curves depending on the number of factors of $N$. We strongly believe that it is possible to prove an analogue of Theorem~\ref{RSconst} for modular curves of the form $X_0(N)$ for general $N$ by the same strategy and modifying Proposition 3.2.2 of Abbes-Ullmo  \cite{MR1437298}  suitably. However, the 
validity of our Theorem~\ref{MaintheoremDDC} depends on the information about special fibers of the arithmetic surface associated to a modular curve and this can't be 
extended to an arbitrary $N$  without non trivial algebro-geometric consideration. Hence, we choose to give a complete proof of Theorem~\ref{RSconst} only for $N=p^2$ in this paper. 
Altough, we write down the computations in the sections \S~\ref{cangreen}, ~\ref{Tracenice} to suit our specific modular curves $X_0(p^2)$ considered 
in this paper, most of the results can be generalized with some minor changes to any modular curve of the form $X_0(N)$ with arbitrary $N \in \N$. For the general strategy  
to prove the theorem for general modular curves, we wish to refer to Abbes-Ullmo  \cite{MR1437298}, 
Mayer  \cite{MR3232770} and M. Grados Fukuda~\cite{Grados:Thesis}.

Being an algebraic curve over $\qq$, $X_0(p^2)$ has a minimal regular model over $\zz$ which we denote 
by $\cX_0(p^2)$. Let $\overline{\omega}_{p^2}$ be the relative dualizing sheaf of $\cX_0(p^2)$ equipped with the 
Arakelov metric and  $\overline{\omega}_{p^2}^2 = \langle \overline{\omega}_{p^2}, \overline{\omega}_{p^2}\rangle$ be 
the Arakelov self-intersection number as defined in Section~\ref{sec:Prelim}.

The following theorem is 
analogous to Proposition D of \cite{MR1437298}, Theorem 1 of \cite{MR3232770} and Theorem 5.2.3 of \cite{Grados:Thesis} for the modular curve $X_0(p^2)$:

\begin{theorem} \label{MaintheoremDDC}
  The Arakelov self intersection numbers for the modular curve $X_0(p^2)$ satisfy the following asymptotic formula
  \begin{equation*}
    \overline{\omega}_{p^2}^2=  3g_{p^2} \log(p^2)+o(g_{p^2} \log(p)).
  \end{equation*}
\end{theorem}

Similar results on the Arakelov self intersection numbers for general arithmetic surfaces are obtained in~\cite{MR3581222}, \cite{Kuhnupperbound}. It is only possible to give an upper and lower bound for general arithmetic surfaces but in the case of modular curves of the form $X_0(p^2)$, we obtain  an asymptotic expression since the recquired algebro-geometric information is available thanks to the work of Bas Edixhoven.

\subsection*{Acknowledgements}
The article owes its existence to  several e-mail communications and encouragements 
of Professors Jurg Kramer, Bas Edixhoven, Robin De Jong, Ulf K\"uhn 
and Shou-Wu Zhang. The origin of this article is  a stimulating conference at the Lorentz Center, 
Leiden University on ``heights and moduli space" in June, 2013. The first named author sincerely 
acknowledges the support of the organizers Professors Robin De Jong, Gerard Freixas and  
Gerard Van Der Geer. The authors are deeply indebted to Dr. Anilatmaja Aryasomayajula to encourage 
them constantly to work on this fascinating problem and for answering several queries on the 
Arakelov Green's functions. The authors wish to express sincere gratitude to the organizers of 
the math symposium at IISER, Pune, 2015 from where the authors started discussing about the project.
Finally, we are deeply indebted to the anonymous referee for numerous suggestions that helped authors 
to improve quality of the article. 

\section{Arakelov intersection pairing} \label{sec:Prelim}
Let $K$ be a number field and and $R$ be its ring of integers. Let $\cX$ be an arithmetic surface over $\spec R$
(in the sense of Liu \cite[Chapter 8, Definition 3.14]{MR1917232}) with the map $f: \cX \to \spec R$. Let $X = \cX_{(0)}$
be the generic fiber which is a smooth irreducible projective curve over $K$. For each embedding $\sigma: K \to \cc$ 
we get a connected Riemann surface $\cX_{\sigma}$ by taking the $\cc$ points of the scheme
\begin{equation*}
  X \times_{\spec K, \sigma} \spec \cc.
\end{equation*}
Collectively we denote 
\begin{equation*}
  \cX_{\infty} = X(\cc) = \bigsqcup_{\sigma: K \to \cc} \cX_{\sigma}.
\end{equation*}

Any line bundle $L$ on $\cX$ induces a line bundle on $\cX_{\sigma}$ which we denote by $L_{\sigma}$. A metrized 
line bundle $\bar{L} = (L, h)$ is a line bundle $L$ on $\cX$ along with a hermitian metric $h_{\sigma}$ on each
$L_{\sigma}$. Arakelov invented an intersection pairing for metrized line bundles which we describe now. 
Let $\bar{L}$ and $\bar{M}$ be two metrized line bundles with 
non-trivial global sections $l$ and $m$ respectively such that the associated divisors do not have any common
components, then
\begin{equation*}
  \langle \bar{L}, \bar{M} \rangle = \langle L, M \rangle_{\fin} + 
                                     \sum_{\sigma: K \to \cc} \langle L_{\sigma}, M_{\sigma} \rangle.
\end{equation*}
The first summand is the algebraic part whereas the second summand is the analytic part of the intersection. 
For each closed point $x \in \cX$, $l_x$ and $m_x$ can be thought of as elements of $\cO_{\cX, x}$ 
via a suitable trivialization. If $\cX^{(2)}$ is the set of closed points of $\cX$, (the number 2 here signifies 
the fact that a closed point is an algebraic cycle on $\cX$ of codimension 2), then 
\begin{equation*}
  \langle L, M \rangle_{\fin} = \sum_{x \in \cX^{(2)}} \log \# (\cO_{\cX,x}/ (l_x,m_x)).
\end{equation*}
Now for the analytic part, we assume that the associated divisors of $l$ and $m$ which we 
denote by $\Div(l)_{\sigma}$ and $\Div(m)_{\sigma}$ on $\cX_{\sigma}$ do not have any common points, 
and that $\Div(l)_{\sigma} = \sum_{\alpha} n_{\alpha} P_{\alpha}$ with $n_{\alpha} \in \zz$, then 
\begin{equation*}
  \langle L_{\sigma}, M_{\sigma} \rangle = -\sum_{\alpha} n_{\alpha} \log || m(P_{\alpha}) ||
          - \int_{\cX_{\sigma}} \log ||l|| c_1(M_{\sigma}).
\end{equation*}
Here $||\cdot||$ denotes the norm given by the hermitian metric on $L$ or $M$ respectively and is clear from 
the context. The first Chern class of $M_{\sigma}$ is denoted by $c_1(M_{\sigma})$ and it is a closed $(1,1)$ form on 
$\cX_{\sigma}$ (see for instance Griffith-Harris \cite{MR1288523}).

This intersection product is symmetric in $\bar{L}$ and $\bar{M}$. Moreover if we consider the group of 
metrized line bundles upto isomorphisms, called the arithmetic Picard group denoted by $\widehat{\pic} \cX$,
then the arithmetic intersection product extends to a symmetric bilinear form on all of $\widehat{\pic} \cX$. It 
can be extended by linearity to the rational arithmetic Picard group 
\begin{equation*}
  \widehat{\pic}_{\qq} \cX = \widehat{\pic} \cX \otimes \qq.
\end{equation*}
For more details see Arakelov \cite{MR0466150, MR0472815} and Curilla \cite{Curilla:Thesis}.

Arakelov gave a unique way of attaching a hermitian metric to a line bundle on $\cX$,
see for instance Faltings \cite[Section 3]{MR740897}. We summarise the construction here. 
Note that the space $H^0(\cX_{\sigma}, \Omega^1)$ of holomorphic differentials on $\cX_{\sigma}$ has 
a natural inner product on it
\begin{equation*}
  \langle \phi, \psi \rangle = \frac{i}{2} \int_{\cX_{\sigma}} \phi \wedge \overline{\psi}. 
\end{equation*}
Let us assume that the genus of $\cX_{\sigma}$ is greater than or equal to $1$. Choose an orthonormal basis
$f_1^{\sigma}, \ldots, f_g^{\sigma}$ of $H^0(\cX_{\sigma}, \Omega^1)$, . The canonical volume form on 
$\cX_{\sigma}$ is 
\begin{equation*}
  \mucan^{\sigma} = \frac{i}{2g} \sum_{j=1}^g f_j^{\sigma} \wedge \overline{f_j^{\sigma}}.
\end{equation*}
There is a hermitian metric on a line bundle $L$ on $\cX$ such that $c_1(L_{\sigma}) = 
\deg(L_{\sigma}) \mucan^{\sigma}$ for each embedding $\sigma: K \to \cc$. This metric is unique
upto scalar multiplication. Such a metric is called admissible. An admissible metric may also be 
obtained using the canonical Green's function, we describe that procedure presently.

Let now $X$ be a Riemann surface of genus greater than $1$ and $\mucan$ the canonical volume form. The canonical Green's function for $X$ is the unique solution of the differential equation
\[
\pa_{z} \pa_{\ov z}\ \Gcan (z,w)=i\pi(\mucan(z)-\de_w(z))
\]
where $\de_w(z)$ is the Dirac delta distribution, with the normalization condition
\[
\int_X \Gcan(z,w) \mucan(z) = 0.
\]
For $Q \in X$ there is a unique admissible metric on $L = \cO_X(Q)$ such that the norm of the constant
function $1$, which is a section of $\cO_X(Q)$, at the point $P$ is given by 
\begin{equation*}
  |1|(P) = \exp(\Gcan(P,Q)).
\end{equation*}
By tensoring we can get an admissible metric on any line bundle on $X$. 

Let again $\cX$ be an arithmetic surface over $R$ as above. Now we assume that the generic genus of $\cX$ is greater 
than 1. To any line bundle $L$ on $\cX$ we can associate in this way a hermitian metric on $L_{\sigma}$ for each 
$\sigma$. This metric is called the Arakelov metric.

Let $L$ and $M$ be two line bundles on $\cX$, we equip them with the Arakelov metrics to get metrized line bundles
$\bar{L}$ and $\bar{M}$. The Arakelov intersection pairing of $L$ and $M$ is defined as arithmetic intersection 
pairing of $\bar{L}$ and $\bar{M}$		
\begin{equation*}
  \langle L, M \rangle_{Ar} = \langle \bar{L}, \bar{M} \rangle.
\end{equation*}
It relates to the canonical Green's function as follows. Let $l$ and $m$ be meromorphic sections of $L$ and $M$ as above.
Assume that the corresponding divisors don't have any common components. Furthermore let
\begin{equation*}
  \Div(l)_{\sigma} = \sum_{\alpha} n_{\alpha, \sigma} P_{\alpha, \sigma}, \quad
  \text{and} \quad 
  \Div(m)_{\sigma} = \sum_{\beta} r_{\beta, \sigma} Q_{\beta, \sigma}
\end{equation*}
then 
\begin{equation*}
  \langle L, M \rangle_{Ar} = \langle L, M \rangle_{\fin}  - \sum_{\sigma: K \to \cc}
                              \sum n_{\alpha,\sigma}r_{\beta,\sigma}\ \Gcan^{\sigma}
                              (P_{\alpha,\sigma}, Q_{\beta,\sigma}).
\end{equation*}

By $\overline{\omega}_{\cX, Ar}$ we denote the relative dualizing sheaf on $\cX$ (see Qing Liu \cite{MR1917232}, 
chapter 6, section 6.4.2) equipped with the Arakelov metric. We shall usually denote this simply by 
$\overline{\omega}$ if the arithmetic surface $\cX$ is clear from the context. 

We are interested in a particular invariant of the modular curve $X_0(p^2)$ which arises from Arakelov geometry
and has applications in number theory. The modular curve $X_0(p^2)$ which is defined over $\qq$ and has a minimal regular model 
$\cX_0(p^2)$ over $\zz$ for primes $p>5$. In this paper we shall calculate the Arakelov self intersection 
$\overline{\omega}^2 = \langle \overline{\omega}, \overline{\omega} \rangle$ of the relative dualizing sheaf 
on $\cX_{0}(p^2)$.

We retain the notation $K$ for a number field and $R$ its ring of integers. If $X$ is a smooth curve over $K$ 
then a regular model for $X$ is an arithmetic surface $p: \cX \to \spec  R$ with an isomorphism of the generic fiber 
$\cX_{(0)}$ to $X$. If genus of $X$ is greater than $1$ then there is a minimal regular model $\cX_{min}$, which 
is unique. $\cX_{min}$ is minimal among the regular models for $X$ in the sense that any proper birational morphism 
to another regular model is an isomorphism. Another equivalent criterion for minimality is that $\cX_{min}$ does not 
have any prime vertical divisor that can be blown down without introducing a singularity.

A regular model for $X_0(p^2)/\qq$ was constructed in Edixhoven~\cite{MR1056773}. We denote this model by 
$\tcX_0(p^2) / \zz$. This model is not minimal but a minimal model $\cX_0(p^2)$ is easily obtained 
by blowing down certain prime vertical divisors. We describe these constructions in Section \ref{sec:MinimalModel}.
\begin{remark}\label{genus-form}
By \cite[Theorem 3.1.1]{MR2112196}, the genus $g_{p^2}$ of $X_0(p^2)$ is given by
\label{genuscomp}
\[
g_{p^2}=1+\frac{(p+1)(p-6)-12c}{12}
\]
where $c \in \{0,\frac{1}{2},\frac{2}{3}, \frac{7}{6}\}$. 
\end{remark}

\section{Canonical Green's functions and Eisenstein series} \label{cangreen}

In this section, we evaluate the canonical Green's function $\Gcan$ for $X_0(p^2)$ at the cusps in terms of the Eisenstein series. 

\subsection{Eisenstein series}
We  recall the definition and some properties of the Eisenstein series that we need for our purpose. For a more elaborate discussion on Eisenstein series for general congruence subgroups, we refer to  K\"uhn~\cite{MR2140212} or Grados~\cite[p. 10]{Grados:Thesis}. Let $\partial(X_0(p^2))$ 
be the set of all cusps of $X_0(p^2)$. 
 By \cite{MR3251709}, we have a complete description of the set of cusps of modular curve $X_0(p^2)$
\[
\partial(X_0(p^2))=\left\{0, \infty, \frac{1}{p}, \ldots, \frac{1}{lp} \right\}, \quad  l=1, \ldots, (p-1). 
\]

For $P \in \partial(X_0(p^2))$,
let $\Ga_0(p^2)_P$ be the stabilizer of $P$ in $\Ga_0(p^2)$.  Denote by $\sigma_P$, any scaling matrix of the cusp $P$, i.e., $\sigma_P$ is an element of $\SL_2(\rr)$ with the properties $\sigma_P(\infty)=P$ and  
\[
\sigma_P^{-1}\Ga_0(p^2)_P \sigma_P=\Ga_0(p^2)_{\infty}=\left\{\pm \left(\begin{array}{cc}
1 & m\\
0 & 1\\
\end{array}\right) \mid m \in \Z \right\}. 
\] 
We fix such a matrix. For $\gamma=\left(\begin{smallmatrix}
a & b\\
c  & d
\end{smallmatrix}\right) \in \SL_2(\rr)$ and $k \in \{0,2\}$, define  the automorphic factor of weight $k$ to be 
\begin{equation}\label{auto-factor}
j_{\ga}(z;k)=\frac{(cz+d)^k}{|cz+d|^k}.
\end{equation}
Let $\tH=\{z| z \in \cc; \Im(z)>0\}$ be the complex upper half plane. 
\begin{definition}
 For $z \in \mathbb{H}$ and $s \in \cc$ with $\Re(s)>1$, the non-holomorphic Eisenstein series $E_{P,k}(z,s)$ at a cusp 
$P \in \partial(X_0(p^2))$ of weight $k$ is defined to be
\[
E_{P,k}(z,s)=\sum_{\ga \in \Ga_0(p^2)_P\backslash \Ga_0(p^2)} \big(\Im(\sigma_P^{-1} \ga z) \big)^s
j_{\sigma_P^{-1} \ga}(z;k)^{-1}.
\]
\end{definition}
The series $E_{P,k}(z,s)$ is a holomorphic function of $s$ in the region $\Re(s)>1$ and for each such $s$, it is an automorphic form (function if $k=0$) of $z$ with respect to $\Ga_0(p^2)$. Moreover, it has a meromorphic continuation to the whole complex plane. Also, $E_{P,k}$ is an eigenfunction of the hyperbolic Laplacian $\De_k$ of weight $k$. Recall that
\begin{equation}\label{hyp-lap}
\De_k=y^2\left(\frac{\pa^2}{\pa x^2} + \frac{\pa^2}{\pa y^2}\right)-ik(k-1) y \frac{\pa}{\pa x}.
\end{equation}
For any $N \in \N$, let $v_{\Ga_0(N)}$ be the volume of the Fuchsian group of first kind $\Ga_0(N)$ \cite[Equation 5.15, p. 183]{MR2112196}. 
We note that $E_{P,0}$ has a simple pole at $s=1$ with residue $1/v_{\GaP}$  \cite[Proposition 6.13]{MR1942691} independent of the $z$ variable.  Being an automorphic function, $E_{P,0}(z,s)$ has a Fourier series expansion at any cusp $Q$, given by
\begin{equation}\label{klf}
E_{P,0}(\sigma_Q(z),s)
=\de_{P,Q}y^s+\phi_{P,Q}^{\Ga_0(p^2)}(s)y^{1-s} +\sum_{n \neq 0} \phi_{P,Q} ^{\Ga_0(p^2)}(n,s) W_s (nz);
\end{equation}
where
\begin{align*}
\phi_{P, Q}^{\Ga_0(p^2)}(s) & = 
\sqrt{\pi} \frac{\Gamma(s-\frac{1}{2})}{\Gamma(s)}\sum_{c=1}^{\infty} c^{-2s}S_{P, Q}(0,0; c),\\
\phi_{P,Q}^{\Ga_0(p^2)}(n,s) & =\pi^s \Ga(s)^{-1} \vert n \vert^{s-1} \sum_{c=1}^{\infty}c^{-2s} S_{P,Q}(0,n;c). 
\end{align*}
Here, $S_{P, Q}(a,b; c)$ is the Kloosterman sum~\cite[p. 48, equation (2.23)]{MR1942691} and $W_s(z)$ is the Whittaker function~\cite[p. 20, equation (1.26)]{MR1942691}. Let
\begin{equation}\label{C-Ga}
\mathcal{C}_{P, Q}^{\Gamma_0(p^2)} = \lim_{s \to 1} \left( \phi_{P, Q}^{\Ga_0(p^2)}(s)-\frac{1}{v_{\Ga_0(p^2)}} \frac{1}{s-1}\right) 
\end{equation}
be the constant term in the Laurent series expansion of $\phi_{P, Q}^{\Ga_0(p^2)}(s)$. 

The following proposition regarding the Eisenstein series of weight zero at the cusp $\infty$ will be crucial in the subsequent sections. 
\begin{proposition}\label{Es1}
The Eisenstein series of weight zero at the cusp $\infty$ can be expressed as
\begin{equation}
\label{Eisenimp}
E_{\infty,0}(z,s)=\frac{1}{2}\frac{1}{\zeta(2s)}\frac{1}{1-p^{-2s}}\left[\sum_{(m,n)}^{\prime}\frac{y^s}{\vert p^2mz+n\vert^{2s}}-\sum_{(m,n)}^{\prime}\frac{y^s}{\vert p^2mz+pn\vert^{2s}}\right].
\end{equation}
Here, $\sum_{(m,n)}^{\prime}$ denote the summation over  $(m,n) \in \zz^2-\{(0,0)\}$. 
\end{proposition}
\begin{proof}
For $t \in \{1,p\}$, notice that 
\begin{align*}
\sum_{(m,n)}^{\prime}\frac{y^s}{\vert p^2mz+tn\vert^{2s}}=\sum_{d=1}^{\infty} \sum_{(m,n)\in \zz^2; (m,n)=d}^{\prime}\frac{y^s}{\vert p^2mz+tn\vert^{2s}}\\
=\sum_{d=1}^{\infty} \sum_{(m',n')\in \zz^2; (m',n')=1}^{\prime}\frac{y^s}{\vert p^2dm'z+td n'\vert^{2s}}\\
=\zeta(2s)\left(\sum_{(m,n)=1} \frac{y^s}{\vert p^2 mz +tn \vert^{2s}}\right).
\end{align*}
In other words, Equation~\eqref{Eisenimp} is equivalent to
\begin{equation}\label{e1}
2(1-p^{-2s})E_{\infty,0}(z,s)=\sum_{(m,n)=1} \frac{y^s}{\vert p^2 mz +n \vert^{2s}} - \sum_{(m,n)=1} \frac{y^s}{\vert p^2 mz + pn \vert^{2s}}.
\end{equation}
The left hand side of (\ref{e1}) is equal to
\begin{equation}\label{lhs}
\begin{aligned}
2(1-p^{-2s}) \sum_{\ga \in \Ga_0(p^2)_{\infty}\backslash \Ga_0(p^2)} \Im(\ga z)^s
  & = 2(1-p^{-2s}) \frac{1}{2}\sum_{\substack{(m,n)=1,\\ m\equiv 0(p^2)}} \frac{y^s}{\vert mz +n \vert^{2s}}\\
  & = \sum_{\substack{(m,n)=1,\\ m\equiv 0(p^2)}} \frac{y^s}{\vert mz +n \vert^{2s}}-p^{-2s}\sum_{\substack{(m,n)=1,\\ m\equiv 0(p^2)}} \frac{y^s}{\vert mz +n \vert^{2s}}\\
  & = \sum_{\substack{(m,n)=1,\\ m\equiv 0(p^2)}} \frac{y^s}{\vert mz +n \vert^{2s}}
    -\sum_{\substack{(m,n)=1,\\ m\equiv 0(p^2)}} \frac{y^s}{\vert pmz + pn \vert^{2s}}\\
  & = \sum_{(p^2m,n)=1}\frac{y^s}{\vert p^2mz +n \vert^{2s}}-\sum_{(p^2m,n)=1}\frac{y^s}{\vert p^3mz + pn \vert^{2s}}.
\end{aligned}
\end{equation}
The first term in the right hand side of (\ref{e1}) is equal to
\begin{align*}
& \sum_{(m,n)=1, p\not |n}\frac{y^s}{\vert p^2 mz +n \vert^{2s}} + \sum_{(m,n)=1, p |n}\frac{y^s}{\vert p^2 mz +n \vert^{2s}}\\
=&\sum_{(m,n)=1, p\not |n}\frac{y^s}{\vert p^2 mz +n \vert^{2s}}+\sum_{(m,pn)=1}\frac{y^s}{\vert p^2 mz + pn \vert^{2s}},
\end{align*}
and the second term in the right hand side of (\ref{e1}) is equal to
\begin{align*}
 & \sum_{\substack{(m,n)=1,\\(m,pn)=1}} \frac{y^s}{\vert p^2 mz
    + pn \vert^{2s}}+\sum_{\substack{(m,n)=1,\\ (m,pn)=p}} \frac{y^s}{\vert p^2 mz + pn \vert^{2s}}\\
  =& \sum_{\substack{(m,n)=1,\\(m,pn)=1}} \frac{y^s}{\vert p^2 mz + pn \vert^{2s}} 
    + \sum_{\substack{(pm,n)=1, \\ (pm,pn)=p}} \frac{y^s}{\vert p^3 mz + pn \vert^{2s}}\\
=  & \sum_{\substack{(m,n)=1,\\ (m,pn)=1}} \frac{y^s}{\vert p^2 mz + pn \vert^{2s}} 
    + \sum_{\substack{(pm,n)=1, \\ (m,n)=1}} \frac{y^s}{\vert p^3 mz + pn \vert^{2s}}.
\end{align*}
Therefore the right hand side of (\ref{e1}) is
\begin{equation}\label{rhs}
\sum_{(m,n)=1, p\not |n}\frac{y^s}{\vert p^2 mz+n\vert^{2s}}-\sum_{(pm,n)=1,(m,n)=1}\frac{y^s}{\vert p^3 mz + pn \vert^{2s}}.
\end{equation}
To complete the proof using \eqref{lhs} and \eqref{rhs}, we only need to observe that $(p^2m,n)=1$ if and only if $(m,n)=1$ and $p\not | n$, so that
\[
\sum_{(p^2m,n)=1}\frac{y^s}{\vert p^2mz +n \vert^{2s}}=\sum_{(m,n)=1, p\not |n}\frac{y^s}{\vert p^2 mz+n\vert^{2s}};
\]
and similarly $(p^2m,n)=1$ if and only if $(pm,n)=1$ and $(m,n)=1$, so that
\[
\sum_{(p^2m,n)=1}\frac{y^s}{\vert p^3mz + pn \vert^{2s}}=\sum_{(pm,n)=1,(m,n)=1}\frac{y^s}{\vert p^3 mz + pn \vert^{2s}}.
\]
\end{proof}

\subsubsection{Computation of $\mathcal{C}_{\infty, \infty}^{\Gamma_0(p^2)}$ and $\mathcal{C}_{\infty,0}^{\Gamma_0(p^2)}$}

In this section, we compute the terms  $\mathcal{C}_{\infty, \infty}^{\Gamma_0(p^2)}$ and $\mathcal{C}_{\infty,0}^{\Gamma_0(p^2)}$ that appear in equation~\ref{gcanimp}. 
To do the same, we expand the constant terms of the Eisenstein series $\phi_{\infty,\infty}^{\Gamma_0(p^2)}(s)$ and $\phi_{\infty,0}^{\Gamma_0(p^2)}(s)$ as defined above. The below computations are inspired by \cite{Keil}.

\begin{lemma} \label{phiinfty}
The Laurent series expansion of $\phi_{\infty,\infty}^{\Gamma_0(p^2)}(s)$ at $s=1$ is given by
\begin{equation*}
\phi_{\infty,\infty}^{\Gamma_0(p^2)}(s)=\frac{1}{v_{\Ga_0(p^2)}} \frac{1}{s-1}+\frac{1}{v_{\Ga_0(p^2)}}\left(2\gamma_{EM} +\frac{a\pi}{6} -\frac{(2p^2-1)\log(p^2)}{p^2-1}\right)+O(s-1),
\end{equation*}
where $\gamma_{EM}$ is the Euler-Mascheroni constant and $a$ is the derivative of $\sqrt{\pi}\frac{\Ga(s-\frac{1}{2})}{\Ga(s) \zeta(2s)}$
at $s=1$.
\end{lemma}

\begin{proof}
From \cite[page 48]{MR1942691}, we compute
\[
S_{\infty, \infty}(0,0;c)=\left|\{ d\pmod c | \left(\begin{smallmatrix}
a & b\\
c  & d
\end{smallmatrix}\right) \in \Ga_0(p^2)\}\right|
= \begin{cases}
0 &  \text{if $ p^2 \nmid c$} ,\\
\phi(c) &  \text{if $p^2 \mid c$}. \\
\end{cases}
\]
Writing $c$ in the form $c=p^{k+2}n$ where $k\geq 0$ and $p \nmid n$,
\[
\phi(c)=\phi(p^{k+2}n)=\phi(p^{k+2})\phi(n)=(p-1)p^{k+1}\phi(n).
\]
Therefore,
\begin{align*}
\sum_{c=1}^{\infty} c^{-2s}S_{\infty,\infty}(0,0;c) &=\sum_{n=1, p\nmid n}^{\infty}\sum_{k=0}^{\infty}(p^{k+2}n)^{-2s}(p-1)p^{k+1}\phi(n)\\
&=p^{-4s+1}(p-1)\sum_{n=1, p\nmid n}^{\infty} n^{-2s}\phi(n)\sum_{k=0}^{\infty} p^{(-2s+1)k}\\
& = p^{-4s+1}(p-1) \left(\frac{\zeta(2s-1)}{\zeta(2s)}\frac{p^{2s}-p}{p^{2s}-1}\right) \left(\frac{1}{1-p^{-2s+1}}\right)\\
&=\frac{p(p-1)}{p^{2s}(p^{2s}-1)}\frac{\zeta(2s-1)}{\zeta(2s)}.
\end{align*}
Hence, we deduce that
\begin{equation}\label{phi-inf-inf-1st-form}
\phi_{\infty,\infty}^{\Gamma_0(p^2)}(s)=\big(p(p-1)\big)\left(\frac{1}{p^{2s}(p^{2s}-1)} \right) \left(\sqrt{\pi}\frac{\Ga(s-\frac{1}{2})}{\Ga(s) \zeta(2s)} \right) \zeta(2s-1).
\end{equation}
The second factor is holomorphic at $s=1$ and has the Taylor series expansion
\begin{equation}\label{TS-p}
\frac{1}{p^{2s}(p^{2s}-1)}= \frac{1}{p^2(p^2-1)} - \frac{ (2p^2-1) \log (p^2) }{p^2(p^2-1)^2}(s-1) + O\big((s-1)^2\big).
\end{equation}
The third factor is holomorphic as well at $s=1$ and has the Taylor series expansion
\begin{equation}\label{TS-Ga}
\sqrt{\pi}\frac{\Ga(s-\frac{1}{2})}{\Ga(s) \zeta(2s)} = \frac{6}{\pi}+a(s-1)+O\big((s-1)^2\big),
\end{equation}
Finally, the Riemann zeta function $\zeta(2s-1)$ is meromorphic at $s=1$ with the Laurent series expansion
\begin{equation}\label{LS-z}
\zeta(2s-1)=\frac{1}{2(s-1)}+\gamma_{EM}+O(s-1).
\end{equation}
Multiplying these expansions, we see that
\[
\phi_{\infty,\infty}^{\Gamma_0(p^2)}(s)=\left(\frac{1}{p(p+1)} \frac{6}{\pi} \frac{1}{2}\right)\frac{1}{s-1} + \left(\frac{1}{p(p+1)} \frac{6}{\pi} \gamma_{EM}+ \frac{1}{p(p+1)} a \frac{1}{2} -\frac{(2p^2-1)\log (p^2)}{p(p+1)(p^2-1)} \frac{6}{\pi}\frac{1}{2} \right)+ O(s-1).
\]
This equation gives the result observing that $v_{\Gamma_0(p^2)}=\frac{\pi}{3}p(p+1)$.
\end{proof}
\begin{corollary}
\label{Secondterminfty}
The constant term in the Laurent series expansion at $s=1$ of the Eisenstein series $E_{\infty,0}$ at the cusp $\infty$ is given by
\[
\mathcal{C}_{\infty, \infty}^{\Gamma_0(p^2)}=\frac{1}{v_{\Ga_0(p^2)}}\left(2\gamma_{EM} +\frac{a\pi}{6} -\frac{(2p^2-1)\log (p^2)}{p^2-1}\right).
\]
\end{corollary}

\begin{lemma}\label{lem-phi-inf-0}
The Laurent series expansion of $\phi_{\infty,0}^{\Ga_0(p^2)}(s)$ at $s=1$ is
\begin{equation}\label{phi-inf-0}
\phi_{\infty,0}^{\Ga_0(p^2)}(s)=\frac{1}{v_{\Ga_0(p^2)}}\frac{1}{s-1} + \frac{1}{v_{\Ga_0(p^2)}} \left(2\gamma_{EM} +\frac{a\pi}{6} -\frac{(p^2-p-1)}{p^2-1} \log (p^2)\right) +O(s-1);
\end{equation}
where $\gamma_{EM}$ and $a$ are as in Lemma \ref{phiinfty}.
\end{lemma}
\begin{proof}
Let $\sigma_0$ be the scaling matrix of the cusp $0$ defined by
\begin{equation}\label{scaling-0}
\sigma_0^{-1}=\frac{1}{\sqrt{p^2}}W_{p^2} \in \SL_2(\mathbf{R}),
\end{equation}
where $W_{p^2} = \left(\begin{smallmatrix}0 & 1\\ -p^2 & 0\end{smallmatrix}\right)$ is the Atkin-Lehner involution. For the cusp $\infty$, we take $\sigma_{\infty}=I$ as a scaling matrix. From \cite[page 48]{MR1942691}, we then have
\begin{align*}
S_{\infty,0}(0,0;c) & =\left\vert \left\{ d\pmod c | \left(\begin{smallmatrix}
a & b\\
c  & d
\end{smallmatrix}\right) = \left(\begin{smallmatrix}pb^{\prime} & -a^{\prime}/p \\ pd^{\prime} & -pc^{\prime}\end{smallmatrix}\right), a^{\prime},b^{\prime},c^{\prime},d^{\prime} \in \mathbf{Z},  a^{\prime}d^{\prime}-b^{\prime}c^{\prime}p^2=1\right\} \right\vert\\
& = \begin{cases}
0 &  \text{if $ p \nmid c$ or $p^2 \mid c$},\\
\phi(n) &  \text{if $c=pn$ with $p\nmid n$}.
\end{cases}
\end{align*}
Therefore, we obtain
\[
\sum_{c} c^{-2s} S_{\infty, 0}(0,0;c) = \sum_{n=1, p \nmid n}^{\infty} p^{-2s}n^{-2s} \phi(n)
=\frac{1}{p^{2s}}\frac{\zeta(2s-1)}{\zeta(2s)}\frac{p^{2s}-p}{p^{2s}-1}.
\]
Thus, we deduce that
\[
\phi_{\infty, 0}^{\Ga_0(p^2)}(s)=(p^{2s}-p) \left(\frac{1}{p^{2s}(p^{2s}-1)}\right) \left(\sqrt{\pi} \frac{\Gamma(s-\frac{1}{2})}{\Gamma(s)\zeta(2s)}\right) \zeta(2s-1) = \frac{1}{p(p-1)} (p^{2s}-p) \phi_{\infty, \infty}^{\Ga_0(p^2)}(s),
\]
using \eqref{phi-inf-inf-1st-form}. The function $p^{2s}-p$ is holomorphic near $s=1$ and has the Taylor series expansion
\begin{equation}\label{TS-p2}
p^{2s}-p=(p^2-p)+(p^2\log (p^2) )(s-1) +O\big((s-1)^2\big).
\end{equation}
Combining this with Lemma \ref{phiinfty} yields (\ref{phi-inf-0}).
\end{proof}

By Lemma \ref{lem-phi-inf-0}, we compute the constant term of the Eisenstein series:
\begin{corollary}
\label{Secondterm}
The constant term of the Eisenstein series $E_{\infty,0}$ at the cusp $0$ is given by
\[
\mathcal{C}_{\infty,0}^{\Gamma_0(p^2)}=\frac{1 }{v_{\Ga_0(p^2)}} \left(2\gamma_{EM} +\frac{a\pi}{6} -\frac{(p^2-p-1)}{p^2-1} \log (p^2)\right).
\]
\end{corollary}
\subsection{Relation between Green's function and Eisenstein series}
Let $G_s(z,w)$ be the {\it automorphic} Green's function 
\cite[p. 4]{MR1437298}. 
The {\it automorphic} Green's function and the Eisenstein series are related by the following equation~\cite[Proposition E, p. 5]{MR1437298}
\begin{align*}
  \Gcan(\infty,0) = & -2\pi \lim_{s \to 1} \left( \phi_{\infty, 0}^{\Ga_0(p^2)}(s)-\frac{1}{v_{\Ga_0(p^2)}}\frac{1}{s-1}\right) 
  -\frac{2\pi}{v_{\Ga_0(p^2)}}\\
    & + 2\pi\lim_{s \to 1} \left(\frac{1}{v_{\Ga_0(p^2)}}\frac{1}{s(s-1)}+\int_{X_0(p^2) \times X_0(p^2)} G_s (z,w)\mu_{\text{can}}(z)     
  \mu_{\text{can}}(w) \right)\\
    & +2\pi \lim_{s \to 1} \left(\int_{X_0(p^2)} E_{\infty,0}(z,s) \mu_{\text{can}}(z) + \int_{X_0(p^2)} E_{0,0}(z,s)   
  \mu_{\text{can}}(z) -\frac{2}{v_{\Ga_0(p^2)}} \frac{1}{s-1} \right).
\end{align*}
For brevity, we write
\begin{equation}\label{R-Ga}
R_{\infty}^{\Gamma_0(p^2)}=\frac{1}{2} \lim_{s \to 1} \left(\int_{X_0(p^2)} E_{\infty,0}(z,s) \mu_{\text{can}}(z) + \int_{X_0(p^2)} E_{0,0}(z,s) \mu_{\text{can}}(z) -\frac{2}{v_{\Ga_0(p^2)}} \frac{1}{s-1} \right). 
\end{equation}
We will show in  Section~\ref{Essterm} that $R_{\infty}^{\Gamma_0(p^2)}$ as defined above coincides with $\mathcal{R}_{\infty}^{\Gamma_0(p^2)}$ that appears in Theorem~\ref{RSconst}. 
From \cite[Proposition 4.1.2, p. 65]{Grados:Thesis},  we know that
\[
2\pi\lim_{s \to 1} \left(\frac{1}{v_{\Ga_0(p^2)}}\frac{1}{s(s-1)}+\int_{X_0(p^2) \times X_0(p^2)} G_s(z,w) \mu_{\text{can}}(z) \mu_{\text{can}}(w) \right)=O\left(\frac{1}{g_{p^2}}\right).
\]
Note that $G_s(z,w)=-G_s^{\Ga_0(p^2)}(z,w)$ in loc. cit. 
The key inputs in the above important estimate are results of Jorgenson-Kramer (\cite[Lemma 3.7, p. 690]{MR2231197} and \cite{MR2521110}) on the constant term of the logarithmic derivate of the Selberg 
zeta function for varying congruence subgroups. 

Since $\frac{g_{p^2}}{v_{\Ga_0(p^2)}}=O(1)$,  we have
\begin{equation}\label{gcanimp}
\Gcan(\infty,0)= -2\pi\mathcal{C}_{\infty,0}^{\Gamma_0(p^2)} +4\pi R_{\infty}^{\Gamma_0(p^2)} + O\left(\frac{1}{g_{p^2}}\right).
\end{equation}

\subsection{Computation of $\mathcal{R}_{\infty}^{\Ga_0(p^2)}$}
\label{Essterm}
We first show that $R_{\infty}^{\Ga_0(p^2)}=\RG$ is the constant term in the Laurent series expansion of the Rankin-Selberg transform at the cusp $\infty$ of the Arakelov metric. We start by writing down the canonical volume form 
$\mucan$ in coordinates. Let 
$S_2\big(\GaP\big)$ be the space of holomorphic cusp forms of weight $2$ with respect to 
$\GaP$. We then have an isomorphism $S_2\big(\GaP\big) \cong H^0(X_0(p^2), \Omega^1)$ given by
 $f(z) \mapsto f(z) dz$. The space of cusp forms  $S_2\big(\GaP\big)$ is equipped with the Petersson inner product. Let $\{f_1, \ldots, f_{g_{p^2}}\}$ 
 be an orthonormal basis of $S_2\big(\Ga_0(p^2)\big)$. The Arakelov metric on $X_0(p^2)$ is the function given by 
\begin{equation}\label{F}
F(z):=\frac{\Im(z)^2}{g_{p^2}} \sum_{j=1}^{g_{p^2}} \vert f_j(z)\vert^2 .
\end{equation}
It is easy to see that the Arakelov metric as defined above is independent of the choice of 
orthonormal basis. In the local co-ordinate $z$, the canonical volume form $\mucan(z)$ is given by \cite[p. 18]{Grados:Thesis}
\begin{equation}\label{mucan-coordinate}
\mucan(z)=\frac{i}{2g_{p^2}} \sum_{j=1}^{g_{p^2}} \vert f_j(z)\vert^2 dz \wedge d\ov z = F(z) \muhyp. 
\end{equation}

Next, we recall the definition of Rankin-Selberg transform of a function at the cusp $\infty$ \cite[p. 19]{Grados:Thesis}.  Let $f$ be a $\Ga_0(p^2)$-invariant holomorphic function on $\mathbb{H}$ of rapid decay at the cusp $\infty$, i.e., the constant term in the Fourier series expansion of $f$ at the cusp $\infty$ 
\[
f(x+iy)= \sum_{n} a_n(y) e^{2\pi i nx}, 
\]
satisfies the asymptotic $a_0(y)=O(y^{-M})$ for some $M>0$ as $y \to \infty$. The Rankin-Selberg transform of $f$ at the cusp $\infty$, denoted by $R_f(s)$, is then defined as
\[
R_f(s):=\int_{Y_0(p^2)} f(z) E_{\infty,0}(z,s) \muhyp(z) = \int_0^{\infty} a_0(y) y^{s-2} \, dy.
\]
The function $R_f(s)$ is holomorphic for $\Re(s) >1$ and admits a meromorphic continuation to the whole complex plane and has a simple pole at $s=1$ with residue
\[
\frac{1}{v_{\Ga_0(p^2)}}\int_{Y_0(p^2)} f(z) \muhyp(z).
\]

The following lemma is similar to \cite[ Lemma 3.5]{MR3232770}. We determine the relation 
between the Rankin-Selberg transforms at the cusp $0$ (defined similarly by replacing $\infty$ with $0$) and $\infty$ from the following lemma. 
\begin{lemma}\label{E_0-E_infty}
\label{eqzi}
The Rankin-Selberg transforms of the Arakelov metric at the cusps $0$ and $\infty$ are related by
\[
\int_{Y_0(p^2)} E_{0,0}(z,s) \mucan(z)=\int_{Y_0(p^2)} E_{\infty,0}(z,s) \mucan(z).
\]
\end{lemma}
\begin{proof}
Let $\sigma_0$ be the scaling matrix of the cusp $0$ defined in (\ref{scaling-0}). Notice that
\begin{multline*}
  E_{0,0}(z,s)  =\sum_{\ga \in \Ga_0(p^2)_0 \backslash \Ga_0(p^2)} \big(\Im(\sigma_0^{-1}\ga z)\big)^s 
    = \sum_{\beta \in \Ga_0(p^2)_{\infty} \backslash \Ga_0(p^2)} \big(\Im(\beta\sigma_0^{-1} z)\big)^s \\
  =\sum_{\beta \in \Ga_0(p^2)_{\infty} \backslash \Ga_0(p^2)} \big(\Im(\sigma_{\infty}^{-1}\beta(\sigma_0^{-1} z))\big)^s 
    = E_{\infty,0}(\sigma_0^{-1}z,s),
\end{multline*}
by taking $\sigma_{\infty}=I$. Now substituting $z=\sigma_0 w$, and using the fact that $\mucan$ is invariant under $\sigma_0$, the result follows.
\end{proof}
As a result we have the following simpler formula for $R_{\infty}^{\Ga_0(p^2)}$ defined in \ref{R-Ga}
\begin{equation}\label{C-1}
R_{\infty}^{\Ga_0(p^2)}=\lim_{s \to 1} \left(\int_{Y_0(p^2)} E_{\infty,0}(z,s) \mu_{\text{can}}(z) -\frac{1}{v_{\Ga_0(p^2)}} \frac{1}{s-1}\right).
\end{equation}
Let $F$ be the Arakelov metric on $X_0(p^2)$ and write $\mu_{\text{can}}(z)=F(z) \muhyp$ in Equation~\eqref{C-1}. 
 Then we see that the integral is the Rankin-Selberg transform $R_F(s)$ of $F$ at the cusp $\infty$ and hence $R_{\infty}^{\Ga_0(p^2)}=\RG$. 
 The next section of the paper is devoted to finding an estimate of $\mathcal{R}_{\infty}^{\Ga_0(p^2)}$ using this formula. 
 In the rest of the paper, we denote by $R_G$ the Rankin-Selberg transform of a function $G$ at the cusp $\infty$. 
 \subsection{Epstein zeta functions}
\label{Quadraticforms}
In this subsection, we define the Epstein zeta functions and recall some basic properties of the same. 
We first recall a connection between quadratic forms and matrices in the group $\Ga_0(N)$. For $a,b,c \in \Z$, 
let $[a,b,c]$ be the quadratic form
\[
\Phi(X, Y)=aX^2+b XY+cY^2=(X,Y) \left(\begin{smallmatrix}
a & b/2\\
b/2 & c\\
\end{smallmatrix}\right)(X,Y)^t.
\]
The discriminant of $\Phi$ is by definition $\dis (\Phi) :=b^2-4ac$.
For any integer $l \in \zz$ with $\vert l \vert \neq 2$, define
\begin{align*}
Q_l & =\{\Phi| \dis (\Phi)=l^2-4\},\\
Q_l(N) & =\big\{\Phi |\Phi \in Q_l ; \Phi=[Na, b, c] : a,b, c \in \Z \big\}.
\end{align*}
The full modular group $\SL_2(\mathbb{Z})$ acts on $Q_l$ by
\begin{equation}\label{mainact}
\begin{aligned}
Q_l \times \SL_2(\Z) & \to Q_l\\
(\Phi, \de) & \mapsto \Phi \circ \de
\end{aligned}
\end{equation}
where $\Phi \circ \de(X,Y)=\Phi\big((X, Y)\de^t\big)$.
If $\delta=\left(\begin{smallmatrix}
x & y\\
z & t\\
\end{smallmatrix}\right)$, then it can be shown that
\[
\Phi \circ \delta=\big[\Phi(x,z), b(x t+y z)+2(axy+czt), \Phi(y,t)\big].
\]
Note that the above defines an action of $\Ga_0(N)$ on $Q_l(N)$. 
\begin{proposition}\label{stabPhi}
Let $\Phi \in Q_l$ be a quadratic form with $\dis (\Phi)=l^2-4$. If $\vert l \vert<2$ then $\SL_2(\mb{Z})_{\Phi}$ is finite and if $\vert l \vert >2$ then
\[
\SL_2(\Z)_{\Phi}=\left\{\pm M^n : n \in \mb{Z}\right\}
\]
for some $M \in \SL_2(\mb{Z})$ with positive trace and which is unique up to replacing $M$ by $M^{-1}$. Moreover, if $\Phi \in Q_l(N)$, then
\[
\Gamma_0(N)_{\Phi}=\SL_2(\Z)_{\Phi}.
\]
\end{proposition}
\begin{proof}
Let $\Phi=[a,b,c]$ be a quadratic form with discriminant $\dis (\Phi)=\De$. If $(x,y) \in \Z^2$ is a solution of the Pell's equation 
$P_{\Delta} : x^2-\Delta y^2=4$, then the matrix 
\[
U_{\Phi}(x,y)=\left(\begin{smallmatrix}
\frac{x-yb}{2} & -cy\\
ay & \frac{x+yb}{2}\\
\end{smallmatrix}\right)
\]
is an automorphism of $\Phi$ and all the automorphisms of $\Phi$ are of this form \cite[p. 63, Satz 2]{MR631688}. Since $\Phi \in Q_l(N)$, we have $U_{\Phi} \in \Ga_0(N)$ and hence $\SL_2(\Z)_{\Phi}=\Gamma_0(N)_{\Phi}$.
\end{proof}

\begin{remark}\label{fd-unit}
For a quadratic form $\Phi \in Q_l$ with $\vert l \vert>2$, let $M$ be as in the proposition. The largest eigenvalue of $M$  is called its fundamental unit and is denoted by $\ep_{\Phi}$. Note that it is well-defined as the eigenvalues of $M^{-1}$ are reciprocals of the eigenvalues of $M$.
\end{remark}

Next, consider the set of matrices in the modular group $\Ga_0(N)$ of trace $l$
\[
\Ga_0(N)_l :=\big\{\ga \in \Ga_0(N) : \tr (\ga) =l\big\}.
\]
The full modular group $\SL_2(\Z)$ acts on this set by conjugation
\begin{align*}
\Ga_0(N)_l \times \SL_2(\Z) & \to \Ga_0(N)_l\\
(\ga, \delta) & \mapsto \de^{-1}\ga \de.
\end{align*}
Note that $\Ga_0(N) \subset \SL_2(\mb{Z})$ also acts on $\Ga_0(N)_l$ in the same way as above.

There is a $\Ga_0(N)$ equivariant one-one correspondence between $\Ga_0(N)_l$ and $Q_l(N)$ given by
\[
\psi: \gamma=\left(\begin{smallmatrix}
a & b\\
N c & d\\
\end{smallmatrix}\right) \mapsto \Phi_{\gamma}=[Nc,d-a,-b],
\]
with inverse
\[
\psi^{\prime}: \Phi=[aN,b,c] \mapsto \gamma_\Phi=\left(\begin{smallmatrix}
\frac{l-b}{2}& -c\\
Na  & \frac{l+b}{2}\\
\end{smallmatrix}\right). 
\]
 We observe that this correspondence induces a bijection
\[
Q_l(N)/\Ga_0(N) \cong \Ga_0(N)_l/\Ga_0(N).
\]

Now fix a quadratic form $\Phi \in Q_l$ and define an action of $\Phi$ on the set $\Z^2$ by 
\begin{align*}
\Phi \cdot (m,n) & =\Phi(n,-m), \quad (m,n) \in \mb{Z}^2.
\end{align*}
Also consider the following subset of $\zz^2$:
\begin{align*}
M^\Phi & =\big\{(m,n) \in \Z^2 | \Phi \cdot (m,n)>0\big\} \subset \zz^2.
\end{align*}
Observe that $\SL_2(\mb{Z})_{\Phi}$ acts on $M^{\Phi}$ by matrix multiplication from the right which follows from the identity
\begin{equation}\label{id-action}
(\Phi \circ \de) \cdot (m,n)\de = \Phi \cdot (m,n),\quad \de \in \SL_2(\mb{Z}).
\end{equation}
\begin{definition}
The Epstein zeta function associated to the quadratic form $\Phi$ is defined to be
\[
\zeta_\Phi(s)=\sum_{(m,n) \in M^\Phi \slash \SL_2(\Z)_\Phi} \frac{1}{(\Phi \cdot(m,n))^s}
\]
which is well-defined by \eqref{id-action}.
\end{definition}
The series converges absolutely for $\Re(s)>1$ and defines a holomorphic function. It has a 
meromorphic extension to the entire complex plane and has a simple pole at $s=1$ with residue (see \cite[\S 3.2.2]{MR1437298}):
\begin{equation}\label{res-ep-zeta}
\Res_{s=1}\zeta_{\Phi}(s)=
\begin{cases}
\frac{2\pi}{\sqrt{\vert \dis(\Phi)\vert}} \frac{1}{\vert  \SL_2(\Z)_\Phi\vert} & \dis (\Phi)<0,\\
\frac{1}{\sqrt{\vert \dis(\Phi)\vert}} \log(\ep_{\Phi}) & \dis \Phi>0,
\end{cases} 
\end{equation}
where $\ep_{\Phi}$ is the fundamental unit as in Remark~\ref{fd-unit}. 
Now let $\Phi \in Q_l(N)$ be a quadratic form and $d\mid N$. Consider the following sets
\begin{align*}
M_d & =\big\{(Nm, dn) \in \mb{Z}^2 \setminus\{0,0\}\big\},\\
M^{\Phi}_d & =\{(Nm, dn) | \Phi \cdot (Nm, dn) >0\}.
\end{align*}
Note that $\SL_2(\mb{Z})_{\Phi}=\Ga_0(N)_{\Phi}$ acts on $M_d^{\Phi}$ in view of \eqref{id-action}.

\begin{definition}
For a quadratic form $\Phi \in Q_l(N)$ and $d \mid N$,  define the zeta function
\[
\zeta_{\Phi,d}(s)=\sum_{(m,n) \in M^{\Phi}_d/\SL_2(\mb{Z})_{\Phi}} \frac{1}{(\Phi \cdot (m,n))^s}
\]
which is well-defined by \eqref{id-action}.
\end{definition}
 
Consider the group homomorphism $*d: \Ga_0(N) \to \Ga_0(d) $ defined by
\[
\ga=\begin{pmatrix}
x & y\\
Nz & t
\end{pmatrix}
\mapsto \ga^{*d}=
\begin{pmatrix}
x & \frac{N}{d}y\\
dz & t
\end{pmatrix}
\]
and note that this map induces the injection $*d : Q_l(N) \to Q_l(d)$
\[
\Phi=[Na,b,c] \mapsto \Phi^{*d}=\left[da, b, \frac{N}{d}c\right].
\]
The map $*d$ respects the action (see \eqref{mainact}) of $\Ga_0(N)$ on $Q_l(N)$ in the sense that $(\Phi \circ \ga)^{*d}=\Phi^{*d} \circ \ga^{*d}$. This immediately implies that $*d$ maps $ \Ga_0(N)_{\Phi}$ into $\Ga_0(d)_{\Phi^{*d}}$. This map is also surjective. Indeed, suppose $\Phi=[Na,b,c]$ and $\De=\dis(\Phi)=\dis(\Phi^{*d})$. If $B\in \Ga_0(d)_{\Phi^{*d}}$, then as in Proposition~\ref{stabPhi},
\[
B=U_{\Phi^{*d}}(x,y)=\begin{bmatrix}
\frac{x-yb}{2} & -\frac{N}{d}cy\\
day & \frac{x+yb}{2}
\end{bmatrix}
\]
where $(x,y)$ is a solution of the Pell's equation $x^2-\De y^2=4$. Set
\[
A=\begin{bmatrix}
\frac{x-yb}{2} & -cy\\
Nay & \frac{x+yb}{2}
\end{bmatrix}.
\]
It is evident that $A \in \Ga_0(N)_{\Phi}$ and $A^{*d}=B$.

To find a relation between $\zeta_{\Phi,d}$ and the Epstein zeta function, note that $\Phi \cdot (Nm, dn)=(Nd) \Phi^{*d}\cdot (m,n)$, and thus
\begin{equation}\label{Phi-Phi*}
\begin{aligned}
\zeta_{\Phi,d}(s) & = \sum_{(m,n) \in M^{\Phi}_d/\Ga_0(N)_{\Phi}} \frac{1}{(\Phi \cdot (m,n))^s}\\
& = \sum_{(m,n) \in M^{\Phi^{*d}}/\Ga_0(d)_{\Phi^{*d}}}\left(\frac{1}{(N d) \Phi^{*d}\cdot (m,n)}\right)^s\\
& =\frac{1}{(Nd)^s}[\SL_2(\Z)_{\Phi^{*d}}:\Ga_0(d)_{\Phi^{*d}}]\zeta_{\Phi^{*d}}(s).
\end{aligned}
\end{equation}
By Proposition~\ref{stabPhi}, we have $[\SL_2(\Z)_{\Phi^{*d}}:\Ga_0(d)_{\Phi^{*d}}]=1$. 
We now calculate the residue of $\zeta_{\Phi,d}(s)$  by expressing it in terms of Epstein zeta functions. 
Our computation is similar to that of \cite[Proposition 3.2.3]{MR1437298}.
 From \eqref{res-ep-zeta}, we get for $\dis (\Phi)>0$
\begin{equation}\label{res-ep*}
\Res_{s=1}\zeta_{\Phi,d}(s)=\frac{ \log \ep_{\Phi^{*d}}}{Nd\sqrt{\dis (\Phi^{*d})}} =\frac{\log \ep_{\Phi}}{Nd\sqrt{\dis (\Phi})}. 
\end{equation}
For $\dis (\Phi)<0$, we also obtain
\begin{equation}\label{res-ep-ell}
\Res_{s=1}\zeta_{\Phi,d}(s)=\frac{2\pi}{Nd\sqrt{\vert \dis (\Phi^{*d})\vert}\vert  \SL_2(\Z)_{\Phi^{*d}}\vert}=\frac{2\pi}{Nd\sqrt{\vert \dis (\Phi)\vert}\vert  \SL_2(\Z)_\Phi\vert}.
\end{equation}

\begin{definition}\label{zeta-Ga}
Let  $\mu$ be  the M\"obius function. Define  the zeta function
\[
\zeta_{\Ga_0(p^2)}(s,l)= \frac{1}{2 \zeta(2s) (1-p^{-2s})}\sum_{d\in\{1,p\}} \mu(d)  
\sum_{\Phi \in Q_l(p^2)/\GaP} \zeta_{\Phi,d}(s).
\]
\end{definition}
Again, these zeta functions are suitable linear combinations of Epstein zeta functions. Hence these functions have meromorphic 
continuations in the entire complex $s$ plane with simple pole at $s=1$. The residues 
can be computed using the formulae of residues of Epstein zeta functions. 
In other words, we have the following Laurent series expansion at $s=1$:
\begin{align*}
\zeta_{\GaP}(s,l) & =\frac{a_{-1}(l)}{s-1}+a_0(l)+O(s-1).
\end{align*}

\section{Spectral expansions of automorphic kernels} \label{Tracenice}
\subsection{Selberg trace formula}
 To compute $R_F(s)$, we follow the strategy carried out in \cite{MR1437298} and \cite{MR3232770}. Consider on one hand the spectral expansions of certain automorphic kernels that consist of the term $F$ together with some well-behaved terms and on the other hand the contribution of various motions of $\Ga_0(p^2)$ to these kernels. Computation of $R_F(s)$ then reduces to understanding the other terms in this identity which are easier to handle. To elaborate this process, let us recall the definitions of these kernels.  For $t>0$, let $h_t: \mathbb{R} \to \mathbb{R}$ be the test  function
\[
h_t(r)=e^{-t(\frac{1}{4}+r^2)},
\]
which is a function of rapid decay. 
\subsubsection{Automorphic kernels of weights $k=0,2$}The automorphic kernels involve the inverse Selberg/Harish-Chandra transform $\phi_k(t, \cdot)$ of $h_t$ of weights $k=0,2$. These are given by
\begin{equation}\label{k_0}
\begin{aligned}
& g_t(v) =\frac{1}{2\pi} \int_{-\infty}^{\infty} h_t(r) e^{-ivr} \, dr, \quad v \in \mathbb{R},\\
& q_t(e^v+e^{-v}-2) = g_t(v), \quad v\in \mathbb{R},\\
& \phi_{0}(t,u) =-\frac{1}{\pi}\int_{-\infty}^{\infty} q_t^{\prime}(u+v^2) \, dv, \quad u \geq 0,\\
& \phi_2(t,u)=-\frac{1}{\pi} \int_{-\infty}^{\infty} q^{\prime}_t(u+v^2) \frac{\sqrt{u+4+v^2} - v}{\sqrt{u+4+v^2} + v} \, dv, \quad u \geq 0.
\end{aligned}
\end{equation}
 Consider the functions
\[
u(z,w)=\frac{\vert z- w \vert^2}{4 \Im z \Im w} \quad
\text{and} \quad H(z,w)=\frac{w-\overline{z}}{z-\overline{w}},
\]
and for $\ga \in \SL_2(\zz)$ set
\[
\nu_k(t,\ga;z)=j_{\ga}(z,k) H^{k/2}(z,\ga z)\phi_k(t,u(z,\ga z)). 
\]
Here,  $j_\ga$ is the automorphic factor defined in (\ref{auto-factor}).

The automorphic kernel $K_k(t,z)$ of weight $k$ with respect to $\GaP$ is defined as
\begin{equation}\label{auto-ker}
\begin{aligned}
K_0(t,z) &:= \frac{1}{2} \sum_{\ga \in  \GaP} \nu_0(t,\ga;z)=\frac{1}{2} \sum_{\ga \in \GaP} \phi_0\big(t, u(z, \ga z)\big) ,\\
K_2(t,z)& :=\frac{1}{2} \sum_{\ga \in  \GaP} \nu_2(t,\ga;z) = \frac{1}{2} \sum_{\ga \in  \GaP} j_{\ga}(z,2)H(z, \ga z) \phi_2\big(t, u(z, \ga z)\big).
\end{aligned}
\end{equation}
We also denote the corresponding summations over the elliptic, hyperbolic, and parabolic elements of  
$\GaP$ by $\sE_k$, $H_k$, and $P_k$ respectively, i.e.,
\begin{equation}\label{auto-ker-parts}
\begin{aligned}
\sE_k(t, z) & :=\frac{1}{2} \sum_{\substack{\ga \in  \GaP\\ \vert \tr (\ga)\vert<2}} \nu_k(t,\ga;z),\\
H_k(t,z) & := \frac{1}{2} \sum_{\substack{\ga \in  \GaP\\ \vert \tr (\ga)\vert>2}} \nu_k(t,\ga;z),\\
P_k(t,z) & := \frac{1}{2} \sum_{\substack{\ga \in  \GaP\\ \vert \tr (\ga)\vert=2}} \nu_k(t,\ga;z),
\end{aligned}
\end{equation}
and write $\sE=\sE_2-\sE_0$, $H=H_2-H_0$ and $P=P_2-P_0$. Let $R_H$, $R_{\sE}$ and $R_P$ be the Rankin-Selberg transforms of these functions at the cusp $\infty$. 

We now simplify the calculation of the Rankin-Selberg transforms of various terms above. To do the same, we introduce
\begin{align}\label{F-R}
F_k^l(t, z) & := \sum_{\substack{\ga \in \GaP\\ \tr (\ga) =l}} \nu_k(t, \ga; z),
\end{align}
\begin{align}\label{F-R1}
R_k^l(t,s) & := \int_{Y_0(p^2)} E_{\infty,0}(z,s) F_k^l(t, z) \muhyp(z).
\end{align}
 We now compute $R_k^l$ by exploiting the connection between 
Epstein zeta function and Eisenstein series.
\begin{lemma}\label{changevariable}
Let $\ga \in \GaP_l$ and suppose $\Phi_\ga \cdot (m,n) >0$.  We have the following equality of integrals
\[
\int_{\mb{H}} \nu_k(t, \ga; z) \frac{y^s}{\vert mz +n \vert^{2s}} \muhyp(z) = \frac{1}{(\Phi_\ga \cdot (m,n))^s} \int_{\mb{H}} \nu_k(t, \ga_{\pm l}; z) y^s \muhyp(z),
\]
where $\gamma_{\pm l} =\left(\begin{smallmatrix}
\pm \frac{l}{2} & \frac{l^2}{4}-1\\
1 & \pm \frac{l}{2} \\
\end{smallmatrix}\right)$.
\end{lemma}
\begin{proof}
For a matrix $\ga=\left(\begin{smallmatrix}
a & b\\
c & d\\
\end{smallmatrix}\right)$, we have the quadratic form $\Phi_{\ga}=[c, d-a, -b]$ associated to $\ga$. Consider the matrix
\[
T=
\frac{1}{\sqrt{\Phi_{\ga} \cdot (m,n)}}\begin{pmatrix}
n & -(d-a)\frac{n}{2}-bm\\
-m & cn - (d-a)\frac{m}{2}
\end{pmatrix}.
\]
We can easily verify that  $T \in \SL_2(\mb{R})$, $T^{-1}\ga T= \ga_l$, $\dfrac{\Im Tw}{\vert mTw+n\vert^2}=\dfrac{\Im w}{\Phi_\ga \cdot (m,n)}$.

A small check using matrix multiplication shows that $\nu_k(t, \ga; Tw)=\nu_k(t, T^{-1}\ga T; w)$. Thus the equality of the integrals follows immediately once we substitute $z=Tw$.
\end{proof}

\begin{proposition}\label{prop:Rkl}
For all $l \in \Z$ with $|l| \neq 2$, we have the following equality
\[
R_k^l(t,s)=\zeta_{\GaP}(s,l)I_k(t,s,l). 
\]
Here, we denote by $I_k^{\pm}(t,s,l) = \displaystyle\int_{\mb{H}}\nu_k(t, \ga_{\pm l}; z) y^s \muhyp(z)$ and $I_k(t,s,l)=I_k^{+}(t,s,l)-I_k^{-}(t,s,l)$.
\end{proposition}

\begin{proof}
From Proposition \ref{Es1}, we obtain
\[
E_{\infty,0}(z,s)= \frac{1}{2 \zeta(2s) (1-p^{-2s})}\sum_{d\in\{1,p\}} \mu(d)\sum_{(m,n) \in M_d} \frac{y^s}{\vert m z +n\vert^{2s}}.
\]
From equations \eqref{F-R} and  \eqref{F-R1}, we now have 
\begin{align*}
R^l_k(t,s) & =\frac{1}{2 \zeta(2s) (1-p^{-2s})}\sum_{d\in\{1,p\}} \mu(d)\int_{Y_0(p^2)} \sum_{\tr (\ga) =l}  \sum_{(m,n) \in M_d}  \nu_k(t, \ga; z) \frac{y^s}{\vert m z+n \vert^{2s}} \muhyp(z).
\end{align*}
For a fixed $l \in \Z$ with $|l| \neq 2$,  denote
\begin{align*}
S_d^{\pm}(l) & =\Big\{\big(\Phi, (m,n) \big) | \Phi \in Q_l(p^2), (m,n) \in M^{\pm \Phi}_d\Big\},
\end{align*}
and 
\[
\sigma^{\pm}_d(l, t,z)  =\sum_{S_d^{\pm}(l)} \nu_k(t, \ga; z) \frac{y^s}{\vert m z+n \vert^{2s}}. 
\]
Recall the identification of the quadratic form $\Phi$ with the matrix $\ga_{\Phi}$  (see \S~\ref{Quadraticforms}). For simplicity of notation,  we write $\ga_{\Phi}=\ga$. For a fixed $\ga$,  we now break the summations inside the integral into two parts to get a simplification 
\[
 \sum_{\tr(\ga) =l}   \sum_{(m,n) \in M_d}  \nu_k(t, \ga; z) \frac{y^s}{\vert m z+n \vert^{2s}} = \sum_{\tr (\ga) =l} ( \sigma_d^{+}(l, t, z)+\sigma_d^{-}(l, t, z)).
\]
The congruence subgroup $\GaP$ acts freely on $S_d^{\pm}$ component wise
\[
\al \cdot \big(\Phi, (m,n)\big) = \big(\Phi \circ \al, (m,n)\al\big).
\]
Hence, we have 
\begin{align*}
R^l_k(t,s) & =\frac{1}{2 \zeta(2s) (1-p^{-2s})}\sum_{d\in\{1,p\}} \mu(d)\int_{Y_0(p^2)} \sum_{\tr (\ga) =l}  
(\sigma_d^{+}(l, t, z)+\sigma_d^{-}(l, t, z)) \muhyp(z).
\end{align*}
Recall that $\Phi \circ \al$ corresponds to $\al^{-1} \ga \al$.  Following \cite[p. 86]{Grados:Thesis}, write $(m_{\al},n_{\al})=(m,n)\al$. As in loc. cit., a small check shows that $\nu_k(t, \al^{-1}\ga\al; z)= \nu_k(t, \ga; \al z)$ and $ \frac{y^s}{\vert m_{\al}z+n_{\al}\vert^{2s}}=\frac{(\Im \al z)^s}{\vert m \al z+n\vert^{2s}}$. We deduce that
\begin{align*}
\sigma_d^{\pm}(l,t,z) & = \sum_{(\Phi, (m,n)) \in S^{\pm}_d/\GaP} \sum_{\al \in \GaP} \nu_k(t, \al^{-1}\ga\al; z) \frac{y^s}{\vert m_{\al}z+n_{\al}\vert^{2s}}\\
& = \sum_{(\Phi, (m,n)) \in S^{\pm}_d/\GaP} \sum_{\al \in \GaP} \nu_k(t, \ga; \al z) \frac{(\Im \al z)^s}{\vert m \al z+n\vert^{2s}}.
\end{align*}
Since the  group $\GaP$ acts on the set $S^{\pm}_d(l)$ component wise, we have
\[
S^{\pm}_d(l)/\GaP = \left\{(\Phi, (m,n)) | \Phi \in Q_l(p^2)/\GaP,  (m,n) \in M^{\pm \Phi}_d/\GaP_{\Phi}\right\}.
\]
Using Lemma~\ref{changevariable}, we deduce that
\begin{align*}
\int_{Y_0(p^2)} \sigma_d^{+}(l,t,z) \muhyp(z) & = \sum_{(\Phi, (m,n)) \in S^{+}_d(l)/\GaP} \int_{\mb{H}} \nu_k(t, \ga; z) \frac{y^s}{\vert m z+n\vert^{2s}} \muhyp(z)\\
&= \sum_{\Phi\in Q_l(p^2)/\GaP} \sum_{(m,n) \in M^{ \Phi}_d/\GaP_{\Phi}} \frac{1}{\Phi_\ga \cdot (m,n)^s} \int_{\mb{H}} \nu_k(t, \ga_{ l}; z) y^s \muhyp(z)\\
& = I_k^{+}(t,s,l) \zeta_{\GaP}(s,l). 
\end{align*}
Observe hat $-\Phi_\ga = \Phi_{-\ga}$ and $S_d^{-}(l)=S_d^{+}(-l)$. Hence, we obtain
\[
\int_{Y_0(p^2)} \sigma_d^{-}(t,z) \muhyp(z)= I_k^{-}(t,s,l) \zeta_{\GaP}(s,l)
\]
and the proposition now follows.
\end{proof}
\subsubsection{Spectral expansions}
Recall that the hyperbolic Laplacian $\De_k$ defined in \eqref{hyp-lap} acts as a positive self-adjoint operator on $L^2(\GaP\backslash \mb{H},k)$---the space of square integrable automorphic forms of weight $k$ \cite[Definition 
3.1.1, p. 22]{MR1437298}. Both these operators have the same discrete spectrum \cite{MR1513277} say
\[
0=\la_0 <\la_1 \leq \la_2 \cdots. 
\]
For weights $k\in \{0,2\}$, the eigenspaces $L^2_{\la_j}(\GaP\backslash \mb{H},k)$ corresponding to eigenvalue $\la_j\neq 0$ are isomorphic via the {\it Maass} operators of weight $0$:
\[
\La_0=iy\frac{\pa}{\pa x}+y\frac{\pa}{\pa y}: L^2_{\la_j}(\GaP\backslash \mb{H},0) \to L^2_{\la_j}(\GaP\backslash \mb{H},2).
\]
We write $\la_j=1/4+r_j^2$ where $r_j$ is real or purely imaginary and let $\{u_j\}$ be an orthonormal basis of eigenfunctions of $\De_0$ corresponding to $\la_j$.

\begin{theorem}{\em \cite[Theorem 1.5.7, p. 16]{Grados:Thesis}}\label{trf}
The spectral expansions are given by
\begin{align*}
K_0(t,z) & =
\frac{1}{v_{\Ga_0(p^2)}}+\sum_{j=1}^\infty h_t(r_j)|u_j(z)|^2 +\frac{1}{4\pi}\sum_{P \in \partial(X_0(p^2))}\int_{-\infty}^{\infty}h_t(r) \left \vert E_{P,0}\left(z,\frac{1}{2}+ir\right)\right\vert^2dr,\\
K_2(t,z) & = g_{p^2} F(z) +\sum_{j=1}^{\infty} \frac{h_t(r_j)}{\la_j} \vert \La_0 u_j (z) \vert^2
+\frac{1}{4\pi} \sum_{P \in \partial(X_0(p^2))} \int_{-\infty}^{\infty} h_t(r) \left\vert E_{P,2}\left(z,\frac{1}{2}+ir\right) \right\vert^2 dr.
\end{align*}
\end{theorem}

\subsection{Different contributions in the Rankin-Selberg transform of the Arakelov metric}
We now prove Proposition~\ref{contb} in this subsection that will determine the term $\RG$ of our Theorerm~\ref{RSconst}. 
To obtain the same, let us first define the following quantities for weights $k \in \{0,2\}$:
\begin{equation}
\begin{aligned}
D_0(t,z) &=\sum_{j=1}^\infty h_t(r_j)|u_j(z)|^2, \quad D_2(t,z)=\sum_{j=1}^{\infty} \frac{h_t(r_j)}{\la_j} \vert \La_0 u_j (z) \vert^2,\\
C_k(t,z) & = \frac{1}{4\pi} \sum_{P \in \partial(X_0(p^2))} \int_{-\infty}^{\infty} h_t(r) \left\vert E_{P,k}\left(z,\frac{1}{2}+ir\right) \right\vert^2 \, dr+\frac{2-k}{2}\frac{1}{v_{\Ga_0(p^2)}};
\end{aligned}
\end{equation}
and write $D=D_2-D_0$ and $C=C_2-C_0$. From Theorem \ref{trf},  we then have
\begin{equation}\label{K2-K0-expn}
\begin{aligned}
K_2(t,z)-K_0(t,z) & =g_{p^2} F(z) +D(t,z) + C(t,z).
\end{aligned}
\end{equation}
On the other hand from \eqref{auto-ker-parts}
\begin{equation}\label{K2-K0-motion}
\begin{aligned}
K_2(t,z)-K_0(t,z)&= H(t,z)+\sE(t,z)+P(t,z).
\end{aligned}
\end{equation}
Combining these two equations, we obtain the identity
\begin{equation}\label{final}
  g_{p^2} F(z)+D(t,z)+C(t,z)=H(t,z)+\sE(t,z)+P(t,z).
\end{equation}
Note that our  equation~\eqref{final} is exactly same as the identity of 
~\cite[p. 72]{Grados:Thesis} as $\gp=P-C$ in loc. cit.   By integrating with respect to $E_{\infty,0}(z,s)\muhyp$,  we obtain the key identity 
\begin{equation}\label{key-form}
g_{p^2}R_F(s)=-R_D(t,s) +R_{H}(t,s) + R_{\sE}(t,s) + R_{P-C}(t,s)
\end{equation}
as announced in the beginning of this section.

\begin{proposition} \label{contb}
The contributions in the Rankin-Selberg transform arising from different motions are given as follows:
\begin{itemize}
\item [(i)] The discrete contribution $R_D(t,s)$ is holomorphic at $s=1$ for all $t$ and $\Rdis_0(t):=R_D(t,1)$ satisfies
\[
\Rdis_0(t)\to 0
\]
as $t \to \infty$.
\item [(ii)]The hyperbolic contribution $R_H(t,s)$ is holomorphic at $s=1$ for all $t$ and $\Rhyp_0(t):=R_H(t,1)$ is of the form
\begin{equation*}
 \Rhyp_0(t) = \Ehyp(t) - \frac{t-1}{2 v_{\GaP}}
\end{equation*}
where $\lim_{t \to \infty} \Ehyp(t) = \frac{1}{2 v_{\GaP}} O_{\epsilon}(p^{2 \epsilon})$.
\item [(iii)] The elliptic contribution $R_{\sE}(t,s)$ at $s=1$ is of the form
\[
R_{\sE}(t,s)=\frac{\Rell_{-1}(t)}{s-1} + \Rell_{0}(t) + O(s-1),
\]
where $\Rell_{-1}(t)$ and $\Rell_{0}(t)$ have finite limits as $t \to \infty$, and $\Rell_0= \displaystyle \lim_{t \rightarrow \infty} \Rell_0(t)=o(\log(p))$.
\item [(iv)] 
The parabolic and spectral contribution
$R_{P-C}(t,s)$ is given by:
\[
R_{P-C}(t,s)= \frac{\Rpar_{-1}(t)}{s-1} + \Rpar_{0} (t) + O(s-1).
\]
Here,  $\Rpar_0(t) = \frac{t+1}{2 v_{\GaP}} + \Epar(t)$ with $\lim_{t \to \infty} \Epar(t) = 
\frac{1-\log(4\pi)}{4\pi} + O\left(\frac{\log p}{p^2}\right)$.
\end{itemize}
\end{proposition}
Note that proof of part (i) of Proposition~\ref{contb} follows from \cite[Lemma 5.2.4]{MR3232770}. 
The explicit description of the modular curves are not used in this portion of 
 \cite{MR3232770} and hence the same proof works verbatim in our case also.
The proof of remaining three parts will be given in the next three subsections.

 As mentioned in the introduction, the underlying strategies for proving the other parts of 
 Proposition~\ref{contb} are also same as that of Abbes-Ullmo~\cite{MR1437298},  Mayer  \cite{MR3232770} and  M. Grados Fukuda~\cite{Grados:Thesis}. However, they used the squarefree assumption of $N$ crucially to write down the Eisenstein series and 
 hence implement the strategy of Zagier \cite{MR633667} involving Selberg's trace formula. We  prove 
Proposition~\ref{Es1} that is analogous to Proposition 3.2.2 of \cite{MR1437298}. The hyperbolic and elliptic contributions are obtained using that.  We made the necessary changes to the computations of the above mentioned papers to suit our specific modular curves $X_0(p^2)$.

 In the rest of this section, 
we simplify the calculation of the Rankin-Selberg transforms of various terms above. For all $l \in \Z$, recall the definitions of  $F_k^l(t, z)$ and $R_k^l(t,s)$ from the  previous section (see ~\eqref{F-R} and ~\eqref{F-R1}).
Then observe that $R_H$, $R_{\sE}$ and $R_P$ can be obtained by summing up $R_2^l-R_0^l$ respectively over $\vert l \vert >2$, $\vert l \vert <2$ and $\vert l \vert =2$. 

We now  prove Proposition~\ref{contb} which is then used to get an expression for
$\RG$. Consequently we derive an asymptotic expression of $\Gcan(\infty,0)$ in terms of $p$.

\subsubsection{The hyperbolic contribution} \label{hyperbolic}
Recall that  the hyperbolic contribution in the Rankin-Selberg transform is determined by a suitable theta function. We proceed to define these theta functions in our context.

Let $\ga \in \GaP$ be a hyperbolic element, i. e., $\tr(\ga)=l; |l|>2$.  Let $v$ the eigenvalue of $\ga$ 
with $v^2>1$. Recall that the norm of the matrix $\ga$ 
is defined to be $N(\ga):=v^2$.
Since $\tr(\ga)=l$, it is easy to see that $n_l=(l+\sqrt{l^2-4})/2$ is the larger eigenvalue of $\ga$.
The theta function for $X_0(p^2)$ is defined by~\cite{MR0473166}
\[
\Theta_{\Ga_0(p^2)}(\xi)=\sum_{|l|>2} \sum_{\Phi \in Q_l(p^2) /\Gamma_0(p^2)} \frac{\log \ep_{\Phi}}{\sqrt{l^2-4}} \frac{1}{\sqrt{4\pi \xi}}e^{-\frac{\xi^2+(\log n_l^2)^2}{4\xi}}.
\]
 Note that this function is exactly equal to the one defined in \cite[p. 54]{MR1437298} as $\Phi \mapsto \ga_{\Phi}$, $N(\ga_0)=\ep_{\Phi}^2$ and $N(\ga)=n_l^2$.
	
\begin{proposition}
The hyperbolic contribution $R_H(t,s)$ in the trace formula  is holomorphic at $s=1$ and
it has the following integral representation:
\begin{align*}
R_{H}(t,1)=-\frac{1}{2v_{\GaP}}\int_0^{t} \Theta_{\Ga_0(p^2)}(\xi) \, d\xi.
\end{align*}
\end{proposition}

\begin{proof}
From Proposition \ref{prop:Rkl}, the hyperbolic contribution is given by 
\[
R_H(t,s)=\sum_{\vert l \vert >2} \zeta_{\GaP}(s,l)\big(I_2(t,s,l)-I_0(t,s,l)\big).
\]
Observe that from \eqref{res-ep*}
\begin{equation}\label{res-zeta-hyp}
\Res_{s=1}\zeta_{\GaP}(s,l) = \frac{1}{\pi v_{\GaP}}  \sum_{\Phi \in Q_l(p^2) \slash \Gamma_0(p^2)} \frac{\log \ep_{\Phi}}{\sqrt{l^2-4}}.
\end{equation}
Let us consider the integral
\[
A_l(t)=-\frac{\pi}{2}\int_0^t \frac{1}{\sqrt{4\pi \xi}}e^{-\frac{\xi^2+(\log n_l^2)^2}{4\xi}} \, d \xi.
\]

From \cite[Propositions 3.3.2, 3.3.3]{MR1437298},  we obtain
\[
I_2(t,s,l)-I_0(t,s,l)=A_l(t)(s-1)+O\big((s-1)^2\big);
\]
for $\Re(s)<1+\de$ and $\de>0$.

It follows that $R_H(t,s)$ is holomorphic near $s=1$ and has the required form.
\end{proof} 

\begin{proof}[Proof of Proposition \ref{contb} (ii)]
By the above proposition, we write
\[
\Rhyp_0 (t) = R_{H}(t,1) = \Ehyp(t) - \frac{t-1}{2v_{\GaP}}
\]
with
\[
\Ehyp(t) = -\frac{1}{2v_{\GaP}}\left(\int_0^t (\Theta_{\GaP}(\xi)-1)\,d\xi +1\right).
\]
Let $Z_{\GaP}$ be the Selberg's zeta function for $X_0(p^2)$.
Note that by \cite[Lemma 3.3.6]{MR1437298}:
\begin{equation}\label{theta-zeta}
\int_0^{\infty} (\Theta_{\GaP}(\xi)-1)\,d\xi= \lim_{s \to 1}\left(\frac{Z_{\GaP}^{\prime}(s)}{Z_{\GaP}(s)} -\frac{1}{s-1}\right)-1. 
\end{equation}
Using \cite[p. 27]{MR1876283}, we obtain
\begin{equation}\label{zeta-lim}
\lim_{s \rightarrow 1} \left(\frac{Z'_{\Ga_0(p^2)}}{Z_{\Ga_0(p^2)}}-\frac{1}{s-1}\right)=O_{\epsilon}(p^{2\epsilon}). 
\end{equation}
It follows that $\lim_{t \to \infty} \Ehyp(t) =\frac{1}{2 v_{\GaP}} O_{\epsilon}(p^{2\epsilon})$ as required. 
\end{proof}

\subsubsection{The elliptic contribution} \label{elliptic}$\phantom{M}$

Recall that (see \S~\ref{Quadraticforms}) there is a bijective correspondence between matrices of traces $l$ and quadratic forms of discriminant $l^2-4$. The explicit map is given by: 
\[
 \Phi=[a,b,c] \mapsto \gamma_\Phi=\left(\begin{smallmatrix}
\frac{l-b}{2}& -c\\
a  & \frac{l+b}{2}\\
\end{smallmatrix}\right). 
\]
To obtain the elliptic contribution in the trace formula, consider matrices with traces 
$l \in \{0, \pm 1\}$. 

 For $D=b^2-4ac=l^2-4$, consider the imaginary quadratic field $K=\qq(\sqrt{-D})$. In this section, we only consider  $D \in \{1,3\}$. For the complex number
$\theta=\frac{b+\sqrt{D}}{2a}$, set $a_{\theta}=\Z+\Z \theta$
and $\cA$ be the ideal class (same as narrow ideal classes for imaginary quadratic fields) corresponding to $a_{\theta}$. For any number field $K$, let $\zeta_K(s)$ be the Dedekind zeta 
function of $K$. 

Let $\zeta(s, \cA)$ be the partial zeta function associated to a (narrow) ideal class $\cA$ \cite[p.131, Appendix C]{Grados:Thesis}.
By \cite[p.78]{Grados:Thesis}, observe that $\zeta_{\Phi}(s)=\zeta(s, \cA)$.
In loc. cit. the author assumed that $|l|>2$. 
Since all the ingredients to write down  this equation
are there in \cite{MR631688} and valid for $|l| \neq 2$, 
it is easy to see that same proof works for $|l| \neq 2$. 
 
In our case,  the class numbers of the quadratic fields $K=\qq(\sqrt{-D})$ are $1$. We have 
an equality of two different zeta functions
 \begin{equation}
 \label{Dedimp}
 \zeta(s,\cA)=\zeta_K(s).
 \end{equation}
 For any $N \in \N$, let $h_l(N)$ be the cardinality of the set $|Q_l(N) \slash \Ga_0(N)|$. In particular for $l \in \{0, \pm 1\}$, we have $h_l(1)=1$ (since the class numbers of the corresponding imaginary quadratic fields are one). We have the following estimate for $h_l(p^2)$:
\begin{align}\label{card-bd}
\big| Q_l(p^2) \slash \Gamma_0(p^2) \big| \leq  \big| Q_l \slash \SL_2(\Z) \big| \big| \SL_2(\Z):\Gamma_0(p^2)\big| \leq h_l(1) \big| \SL_2(\Z):\Gamma_0(p^2) \big| \leq \big| \SL_2(\Z):\Gamma_0(p^2) \big|.
\end{align}

\begin{proof}[Proof of Proposition \ref{contb} (iii)]
The elliptic contribution is obtained by putting $l=0, 1, -1$ in Proposition~\ref{prop:Rkl}. We have
  \begin{align*}
    R_{\sE}(t,s) = & \sum_{l \in \{0,1,-1\}}\big(I_2(t,s,l)-I_0(t,s,l)\big) \zeta_{\GaP}(s,l). 
  \end{align*}
  Recall the Laurent series expansions of the following functions [see  Definition~\ref{zeta-Ga} and \cite[\S 3.3.3, p. 57]{MR1437298}]:
  \begin{align*}
\zeta_{\GaP}(s,l) & =\frac{a_{-1}(l)}{s-1}+a_0(l)+O(s-1),\\
I_2(t,s,l)-I_0(t,s,l) & = b_0(t,l)+b_1(t,l)(s-1)+O((s-1)^2).
\end{align*}
Observe that
\begin{align*}
\Rell_{-1}(t) & =\sum_{l=-1}^{1}a_{-1}(l)b_0(t,l),\\
\Rell_{0}(t) & = \sum_{l=-1}^{1} a_0(l)b_0(t,l)+a_{-1}(l)b_1(t,l).
\end{align*}
Note that $b_i(t,l)$ differ from $C_{i,l}(t)$ of \cite[\S 3.3.3, p. 57]{MR1437298} by a multiplicative function that does not depend on $t$ and  therefore $\lim_{t \to \infty} b_i(t,l)$ exists, say $b_i(\infty,l)$.

We now proceed to find an estimate on $\Rell_{0}(t) $. From Definition~\ref{zeta-Ga} and \eqref{Phi-Phi*}, we obtain
\begin{align*}
\zeta_{\GaP}(s,l) &= \frac{1}{2 \zeta(2s) (1-p^{-2s})}\sum_{d\in\{1,p\}} \mu(d)  \frac{1}{(p^2d)^s}\sum_{\Phi \in Q_l(p^2)/\GaP} \zeta_{\Phi^{*d}}(s)\\
&= \frac{1}{2 \zeta(2s) (1-p^{-2s})} \frac{1}{p^{2s}} \left(
\sum_{\Phi \in Q_l(p^2)/\GaP} \zeta_{\Phi^{*1}}(s)-\frac{1}{p^s}\sum_{\Phi \in Q_l(p^2)/\GaP} \zeta_{\Phi^{*p}}(s)\right)\\
&=I_1(s) I_2(s)  
\end{align*}
with 
$I_1(s)=  \frac{1}{2 \zeta(2s) (1-p^{-2s})} \frac{1}{p^{2s}}$
and 
\[
I_2(s)=\sum_{\Phi \in Q_l(p^2)/\GaP} \zeta_{\Phi^{*1}}(s)-\frac{1}{p^s}\sum_{\Phi \in Q_l(p^2)/\GaP} \zeta_{\Phi^{*p}}(s).
\]

For $t \in \{1,p\}$, consider the following series (depending on $l$):
\begin{align*}
J_t(s)=\sum\limits_{\Phi \in Q_l(p^2)\slash \GaP}  \zeta_{\Phi^{*t}}(s).
\end{align*}

To study the Laurent series expansion of the above series, write
\begin{align*}
J_t(s)= \frac{c_{-1,t}}{s-1}+c_{0,t} +O((s-1)).
\end{align*}
For the Dedekind zeta functions $\zeta_K(s)$ with $K=\qq(\sqrt{l^2-4})$ and $l \in \{0,\pm 1\}$, the residues and the constants of the Laurent series expansions
depend only on the fields $K$. 
For $i \in \{0,-1\}$ and $t \in \{1,p\}$, we then get the following bounds of the coefficients of the above expansion using the estimate~\eqref{card-bd}:
\[
\vert c_{i,t}\vert \leq  C_i\cdot\,  \big|\SL_2(\Z):\Gamma_0(p^2)\big|
\]
for some constants $C_i$'s. These constants are determined by the Dedekind zeta functions and are independent of the prime $p$. 
The corresponding  Laurent series expansion is given by 
\begin{align*}
I_2(s)
& =J_1(s)-\frac{1}{p^s} J_p(s) \\
 & =
 \frac{c_{-1,1}}{s-1}+c_{0,1}+O(s-1)-\frac{1}{p^s}\left( \frac{c_{-1,p}}{s-1}+c_{0,p}+O(s-1)\right)\\
  & =
   \frac{c_{-1,1}}{s-1}+c_{0,1}+O(s-1)-\left(\frac{1}{p}-\frac{\log p}{p}(s-1)+O((s-1)^2)\right)\left( \frac{c_{-1,p}}{s-1}+c_{0,p}+O(s-1)\right)\\
    & =
  \frac{1}{s-1}\left(c_{-1,1}-\frac{c_{-1,p}}{p}\right)+\left(c_{0,1}-\frac{c_{0,p}}{p}+\log p\frac{c_{-1,p}}{p}+O(s-1)\right)\\
    & =
  \frac{A_{-1}(p)}{s-1}+A_0(p)+O(s-1);
\end{align*}
with $A_{-1}(p)=c_{-1,1}-\frac{c_{-1,p}}{p}$ and $A_0(p)=c_{0,1}-\frac{c_{0,p}}{p}+\log p\frac{c_{-1,p}}{p}$. 
A small check shows that:
\begin{align*}
I_1(s) =\frac{3}{\pi^2(p^2-1)} + 
\left(\frac{D_1\log p +D_2}{p^2-1}-\frac{6\log p}{\pi^2(p^2-1)}\right)(s-1)+O(s-1)^2; 
\end{align*}
where the constants $D_1$ and $D_2$ are independent of $p$. We conclude that
\begin{align*}
\zeta_{\Ga_0(p^2)}(s,l) &= \frac{1}{2 \zeta(2s) (1-p^{-2s})}\sum_{d\in\{1,p\}} \mu(d)  \frac{1}{(p^2d)^s}\sum_{\Phi \in Q_l(p^2)/\GaP} \zeta_{\Phi^{*d}}(s)\\
&= I_1(s) I_2(s) \\
&=\frac{3 A_{-1}(p)}{\pi^2(p^2-1)} \frac{1}{s-1}+\frac{3 A_0(p)}{\pi^2(p^2-1)}+\left(\frac{D_1\log(p) +D_2}{p^2-1}-\frac{6\log p}{\pi^2(p^2-1)}\right)A_{-1}(p)+O((s-1)). 
\end{align*}
Hence, we obtain $a_{-1}(l)=\frac{3 A_{-1}(p)}{\pi^2(p^2-1)}$  and $a_0(l)=\frac{3 A_0(p)}{\pi^2(p^2-1)}+\left(\frac{D_1\log p +D_2}{p^2-1}-\frac{6\log p }{\pi^2(p^2-1)}\right)A_{-1}(p)$. 
Observe that all the terms in the above expression have $p^2-1$ in the denominator and it 
is of same order in $p$ as $ |\SL_2(\Z):\Gamma_0(p^2)|$. Since the constants 
$D_1$  and $D_2$  are independent of $p$, we get $\Rell_0 =o( \log p)$. 
\end{proof}

\subsubsection{The parabolic and spectral contribution} \label{parabolic}

Recall that we are interested in the term $R_{P-C}$ of \eqref{key-form}. 
Let $a_0(y,s;q,\infty,k)$ be the zero-th Fourier coefficient of $E_{q,k}(z,s)$. It is given by
\begin{align*}
a_0\left(y,s;q,\infty,0\right) & =\de_{q,\infty} y^s +\phi_{q,\infty}(s)y^{1-s},\\
a_0\left(y,s;q,\infty,2\right) & = \de_{q,\infty}y^s+\phi_{q,\infty}(s)\frac{1-s}{s} y^{1-s}.
\end{align*}

Consider the series
\[
\ti{E}_{q,k}(z,s)=E_{q,k}(z,s)-a_0\left(y,\frac{1}{2}+ir;q\infty,k\right). 
\]

For any $z \in \hh$ with $z=x+iy$,  define the functions
\begin{align*}
  p_1(t,y,k) & = \frac{1}{2}\int_{-1/2}^{1/2} 
               \sum_{\substack{\ga \in \Ga_0(p^2),\\ \vert \tr(\ga)\vert=2, \ga \not \in \GaP_{\infty}}} 
               \nu_k(t, \ga; z) dx, \\
  p_2(t,y,k) & =\frac{1}{2} \int_{-1/2}^{1/2} \sum_{\ga \in \GaP_{\infty}} 
               \nu_k(t,\ga; z) \, dx \ 
               - \ \frac{y}{2\pi}\int_{-\infty}^{\infty} h_t(r) \, dr, \\
  p_3(t,y,k) & =-\frac{y}{2\pi} \int_{-\infty}^{\infty} h_t(r) \phi_{\infty,\infty}\left(\frac{1}{2}-ir\right)
               \left(\frac{\frac{1}{2}+ir}{\frac{1}{2}-ir}\right)^{k/2}y^{2ir} \,dr 
               - \frac{2-k}{2} \frac{1}{v_{p^2}},\\
  p_4(t,y,k)  & = -\frac{1}{4\pi} \sum_{q \in \partial X_0(p^2)} \int_{-1/2}^{1/2} 
                \int_{-\infty}^{\infty} h_t(r)  \left\vert 
                \ti{E}_{q,k}\left(x+iy, \frac{1}{2}+ir\right)\right\vert^2\, dr dx. 
\end{align*}
For $j \in \{1,2,3,4\}$, the corresponding Mellin transforms are defined as
\[
\M_j(t,s)=\int_0^{\infty} \big(p_j(t,y,2)-p_j(t,y,0)\big) y^{s-2}dy.
\]

With our assumption on the prime $p$, observe that $g_{p^2} > 1$. By \cite[Lemma 4.4.1]{Grados:Thesis}, we have for $\Re(s) >1$
\begin{equation}\label{p-s-c}
R_{P-C}(t,s) = \M_1(t,s)+\M_2(t,s)+\M_3(t,s)+\M_4(t,s). 
\end{equation}
To study the function $p_1(t,y,k)$,  we examine the matrices 
that appear in the sum. 
Say $\GaInf =\Ga_0(N)_{\infty}=  \left\{\pm\smmat{1}{m}{0}{1} \mid m \in \zz\right\}$ be the parabolic subgroup of $\Ga_0(N)$.
A simple computation  \cite[p. 37]{MR1437298} involving matrices shows that any matrix  in $\SL_2(\zz)$ of trace $2$ is of the form 
$ \left(\begin{smallmatrix}
1-a & b\\
-c & 1+a\\
\end{smallmatrix}\right)$ with $a^2=bc$.
Matrices with traces $-2$ can be treated in a similar manner by multiplying by $-I$. 
 Write the matrices as 
\[
\left(\begin{smallmatrix}
1-a & b\\
-c & 1+a\\
\end{smallmatrix}\right)=\left(\begin{smallmatrix}
\delta & -\beta\\
-\gamma & \alpha\\
\end{smallmatrix}\right)\left(\begin{smallmatrix}
1 & m\\
0 & 1\\
\end{smallmatrix}\right)\left(\begin{smallmatrix}
\alpha & \beta\\
\gamma & \delta\\
\end{smallmatrix}\right)=\left(\begin{smallmatrix}
1-m\gamma \delta & m \delta^2\\
-m \gamma^2 & 1+m\gamma \delta\\
\end{smallmatrix}\right) \in \Ga_0(p^2). 
\]
To ensure the same we need  $p^2 \mid m \gamma^2$. Hence, we have three possibilities (i)  $p \mid \gamma$ but 
$p \nmid m$ (ii)
 $p \mid m$ but $p^2 \nmid m$, $p \mid \gamma$, and (iii) $p^2 \mid m$. 
In other words, any matrix $\ga \in \GaP$ with trace $2$ is of the form $\sigma^{-1}\smmat{1}{m}{0}{1} \sigma$ with $m\in\mb{Z}$ and for some unique matrix $\sigma$. 
In the first two cases,  $\sigma \in \GaInf \backslash \Ga_0(p)$  where as in the third case 
$\sigma \in \GaInf \backslash \SL_2(\mb{Z})$. 

For weights $k \in \{0,2\}$ and $d \in \{1,p\}$, define the following functions that appear as sub-sums 
in   $p_1(t,y,k)$ corresponding to the cases (i) and (ii):
\begin{align*}
q_d(t,y,k)& = \frac{1}{2}\int_{-1/2}^{1/2} \ \sum_{m \neq 0, (m,p)=1} \ 
\sum_{\substack{\sigma \in \GaInf \backslash \Gamma_0(p),
\\ \sigma \neq \GaInf }} \nu_k\left(t, \sigma^{-1}
\left(\begin{smallmatrix}
\pm 1 & md\\
0 & \pm1\\
\end{smallmatrix} \right)\sigma; z\right)dx \\
& =\int_{-1/2}^{1/2} \ \sum_{m \neq 0, (m,p)=1} \ \sum_{\substack{\sigma \in \GaInf \backslash \Ga_0(p) \slash \GaInf,\\  
\sigma \neq \GaInf}} \,
\sum_{n=-\infty}^{\infty}\nu_k\left(t, \sigma^{-1}
\left(\begin{smallmatrix}
1 & md\\
0 & 1\\
\end{smallmatrix} \right)\sigma; z+n \right)dx.
\end{align*}

Next, we consider the other sub-sum
that appears in $p_1(t,y,k)$ corresponding to case (iii):
\begin{align*}
q_{p^2}(t,y,k)& = \frac{1}{2}\int_{-1/2}^{1/2} \sum_{m \neq 0} 
\sum_{\substack{\sigma \in \GaInf \backslash  \SL_2(\Z), \\ \sigma \neq \GaInf}} \nu_k \left(t, \sigma^{-1}
\left(\begin{smallmatrix}
\pm 1 & mp^2\\
0 & \pm 1\\
\end{smallmatrix} \right)\sigma; z \right)dx \\ 
& =\int_{-1/2}^{1/2} \sum_{m \neq 0} \, 
\sum_{\substack{\sigma \in \GaInf \backslash \SL_2(\Z) \slash \GaInf, \\ 
\sigma \neq \GaInf}}\, \sum_{n=-\infty}^{\infty}\nu_k \left( \sigma^{-1}
\left(\begin{smallmatrix}
1 & mp^2\\
0 & 1\\
\end{smallmatrix} \right)\sigma; z+n \right)dx.
\end{align*}
From the above observation, we have a following decomposition of the function $p_1$:
\[
p_1(t,y,k) = q_1(t,y,k)+ q_p(t,y,k)+q_{p^2}(t,y,k).
\]
Recall that $\GaInf =\Ga_0(N)_{\infty}=  \left\{\pm\smmat{1}{m}{0}{1} \mid m \in \zz\right\}$ is the parabolic subgroup of $\Ga_0(N)$.
\begin{lemma} \label{parasecond}
For any matrix $\tau = \smmat{a}{b}{c}{d} \in \SL_2(\rr) - \GaInf$ with $\tr (\tau)=\pm 2$, we have
\[
\int_{\tH} \nu_k(t,\tau;z) \Im(z)^s \muhyp(z) =
\frac{1}{|c|^s}\int_{\tH} \nu_k(t, L_{\pm};z) \Im(z)^s \muhyp(z);
\]
where $L_{\pm} = \smmat{\pm 1}{0}{1}{\pm 1}$.
\end{lemma}

\begin{proof}
Substitute $z=Tw$ where
$T=\dfrac{1}{\sqrt{c}}\left(\begin{smallmatrix}
1 & \frac{a-d}{2}\\
0& c\\
\end{smallmatrix}\right)$.
\end{proof}
\noindent For any positive integer $M$, define
\[
\L_M(s)=\sum_{\substack{\sigma \in \GaInf \backslash \Gamma_0(M) \slash  \GaInf, 
\\ \sigma=\left(\begin{smallmatrix}
* & *\\
c & *\\
\end{smallmatrix} \right),\ c \neq 0}} \frac{1}{|c|^{2s}}
\]
and 
\[
\zeta_M(s)=\sum_{m \geq 1, (m,M)=1} \frac{1}{m^s}. 
\]
For $k \in \{0,2\}$, consider the following integrals:
\[
I_k(t,s,2)=\int_{\tH} \nu_k(t,L_+;z) (\Im z)^s \muhyp(z)+\int_{\tH} \nu_k(t,L_-;z)(\Im z)^s \muhyp(z).
\]
\begin{proposition}\label{M_1}
For the modular curves of the form $X_0(p^2)$, the Mellin transform $\M_1(t,s)$ can be written as 
 product of following simple functions
\[
\M_1(t,s)=\frac{p}{p-1}  \left(1-\frac{1}{p^{2s}}\right) \zeta(s)\mathcal{L}_p(s)\big(I_2(t,s,2)-I_0(t,s,2)\big). 
\]

\end{proposition}
\begin{proof}
For any matrix $\sigma \in \GaInf \backslash \GaP$ with $\sigma = \left(\begin{smallmatrix}
a & b\\
c & d\\
\end{smallmatrix}\right)$, denote $c(\sigma)=c$. Note that this choice is independent of the 
right coset representative. 

By Lemma \ref{parasecond}, we obtain
\begin{align*}
  \int_0^{\infty} q_1(t,y,k) y^{s-2} dy &= \sum_{\substack{m \neq 0,\\ (m,p)=1}} \ 
    \sum_{\substack{\sigma \in \GaInf \backslash \Gamma_0(p) / \GaInf, \\ \sigma \neq \GaInf}} 
    \int_{\hh} \nu_k(t,\sigma^{-1} \smmat{1}{m}{0}{1}\sigma; z) (\Im z)^{s} \muhyp \\
    &= \sum_{\substack{m \neq 0,\\ (m,p)=1}} \ 
    \sum_{\substack{\sigma \in \GaInf \backslash \Gamma_0(p) / \GaInf, \\ \sigma \neq \GaInf}} 
  \frac{1}{|c(\sigma^{-1} \smmat{1}{m}{0}{1}\sigma)|^s}\ I_k(t,s,2)\\
   & = \zeta_p(s)  I_k(t,s,2) \L_p(s),\\
  \int_0^{\infty} q_p(t,y,k) y^{s-2} dy
    &= \sum_{\substack{m \neq 0,\\ (m,p)=1}}\ 
    \sum_{\substack{\sigma \in \GaInf \backslash \Gamma_0(p) / \GaInf, \\ \sigma \neq \GaInf}} \ 
    \int_{\hh}  \nu_k(t,\sigma^{-1}\smmat{1}{mp}{0}{1} \sigma; z) (\Im z)^{s} \muhyp\\
  &= \frac{\zeta_p(s)}{p^s} I_k(t,s,2) \L_p(s),\\
  \int_0^{\infty} q_{p^2}(t,y,k) y^{s-2} dy &= \sum_{m \neq 0} \ 
    \sum_{\substack{\sigma \in \GaInf \backslash \SL_2(\Z) / \GaInf,\\ \sigma \neq \GaInf}}\ 
    \int_{\hh} \nu_k(t,\sigma^{-1} \smmat{1}{mp^2}{0}{1} \sigma; z) (\Im z)^{s} \muhyp \\
  &=\frac{\zeta(s)}{p^{2s}} I_k(t,s,2) \L_1(s).
\end{align*}
Recall the identity $\zeta_p(s)=\zeta(s)(1-p^{-s})$. 
By summing up, we deduce that
\begin{equation}\label{m_1-form}
\M_1(t,s)=\left(\left(1-\frac{1}{p^{2s}}\right) 
        \L_p(s)+\frac{1}{p^{2s}} \L_1(s)\right)  \zeta(s)\big(I_2(t,s,2)-I_0(t,s,2)\big).
\end{equation}
By \cite[p. 49]{MR1942691} and  \cite[Theorem 2.7, p. 46]{MR1942691}, we now get
\[
\L_1(s)=\frac{\zeta(2s-1)}{\zeta(2s)}.
\]

From \cite[Lemma 3.2.19]{MR1437298}, we deduce that
\begin{equation}\label{impac0}
\phi_{\infty, \infty}^{\Gamma_0(p)}(s)=\sqrt{\pi} \frac{\Gamma(s-\frac{1}{2})}{\Gamma(s)}
\frac{p-1}{p^{2s}-1}\frac{\zeta(2s-1)}{\zeta(2s)}.
\end{equation}
Hence, we obtain  \cite[p.56]{MR1437298}
\begin{equation}\label{impac}
\phi_{\infty, \infty}^{\Gamma_0(p)}(s)=\sqrt{\pi} \frac{\Gamma(s-\frac{1}{2})}{\Gamma(s)}\L_p(s).
\end{equation}
Comparing \eqref{impac0} and \eqref{impac}, we have $\L_1(s)=\dfrac{p^{2s}-1}{p-1}\L_p(s)$. 
The result follows by substituting this in \eqref{m_1-form}.
\end{proof}
   
\begin{proposition} \label{prop:M1}
The Laurent series expansion of $\M_1(t,s)$ at $s=1$ is given by
\begin{align*}
 \M_1(t,s) = & \left( \frac{1}{v_{\Ga_0(p)}}\frac{p+1}{p}A_1(t) \right)\frac{1}{s-1} \\ 
             & + \frac{1}{v_{\Ga_0(p)}} \frac{p+1}{p} 
             \bigg(\left(3\gem + \frac{a\pi}{6} - \log (p^2)\right) A_1(t) + B_1(t) \bigg) \\ 
             & + O(s-1);
\end{align*}
where the functions $A_1(t)$ and $B_1(t)$ are independent of $p$. Furthermore, we have $\lim_{t \to \infty} A_1(t) = 1/2$
and $B_1(t)$ has a finite limit as $t \to \infty$, which we call $B_1(\infty)$.
\end{proposition}

\begin{proof}
From Proposition \ref{M_1} and \eqref{impac}, we obtain
\begin{align}
\label{M1}
\M_1(t,s)  & = \frac{p}{p-1} \left(1-\frac{1}{p^{2s}}\right) \zeta(s)\phi_{\infty, \infty}^{\Ga_0(p)}(s) 
        \left(\frac{\Ga(s)}{\sqrt{\pi}\Ga(s-\frac{1}{2})} \big(I_2(t,s,2)-I_0(t,s,2)\big)\right). 
\end{align}
By \cite[p.59]{MR1437298} we have
\[
  \phi_{\infty, \infty}^{\Ga_0(p)}(s)=
  \frac{1}{v_{\Ga_0(p)}}\frac{1}{s-1} + 
  \frac{1}{v_{\Ga_0(p)}}\left(2 \gem + a \frac{\pi}{6}-\frac{p^2}{p^2-1} \log(p^2)\right) + 
  O(s-1),
\]
and
\[
  \zeta(s)=\frac{1}{s-1} + \gem + O(s-1).
\]
Recall that we have a well-known Laurent series expansion
\begin{equation*}
  \left(1-\frac{1}{p^{2s}}\right) = \left(1-\frac{1}{p^2}\right) + \frac{\log(p^2)}{p^2}(s-1)+ O((s-1)^2).
\end{equation*} 
By a verbatim generalization of \cite[Lemma~B.2.1]{Grados:Thesis} and \cite[Proposition~3.3.4]{MR1437298} we have:
\[
  \frac{\Ga(s) } {\sqrt{\pi}\Ga(s-\frac{1}{2})} [I_2(t,s,2)-I_0(t,s,2)]
  = A_1(t) (s-1)+B_1(t) (s-1)^2+O((s-1)^3);
\]
where $\lim_{t \to \infty}A_1(t)=1/2$. It also follows from \cite[Lemma~3.3.10]{MR1437298} that $B_1(t)$ has a finite
limit as $t \to \infty$.
Putting these in Equation~\eqref{M1}, we get the required result. 
\end{proof}

From \cite[Lemma 4.4.7-9]{Grados:Thesis}, we obtain
\begin{equation}
\begin{aligned} \label{eq:M2}
  \M_2(t,s) & = \frac{1}{s-1}\left( A_2(t)+ \frac{1}{4\pi} \right) 
                + \left( \dfrac{1-\log(4\pi)}{4\pi}+\gem A_2(t)+B_2(t) \right) 
                + O(s-1),\\
  \M_3(t,s) & = \frac{1}{v_{\GaP}}\frac{1}{s-1}  
                + \left( \frac{\C_{\infty, \infty}^{\GaP}}{2}+\frac{t+1}{2 v_{\GaP}} \right)
                + O(s-1),\\
  \M_4(t,s) & = A_4(t)+O(s-1),
\end{aligned}
\end{equation}
where $A_2$, $B_2$, $A_4$ depend only on $t$ and tends to zero as $t \rightarrow \infty$, and 
$\C_{\infty, \infty}^{\GaP}$ is the constant term of $\phi_{\infty, \infty}^{\GaP}$ (see
Lemma~\ref{phiinfty}).

\begin{proof}[Proof of Proposition \ref{contb} (iv)] 
Combining \eqref{p-s-c}, Proposition~\ref{prop:M1} and \eqref{eq:M2}, we get
\begin{equation*}
 R_{P-C}(t,s) = \frac{\Rpar_{-1}(t)}{s-1} + \Rpar_0(t) + O(s-1);
\end{equation*}
where 
\begin{equation*}
\Rpar_{-1} = \frac{1}{4\pi} + \frac{1}{v_{\GaP}} + \frac{1}{v_{\Ga_0(p)}}\frac{p+1}{p}A_1(t) + A_2(t)
\end{equation*}
and
\begin{align*}
\Rpar_0 = & \frac{1}{v_{\Ga_0(p)}} \frac{p+1}{p} 
                       \bigg(\left(3\gem + \frac{a\pi}{6} - \log (p^2)\right) A_1(t) + B_1(t) \bigg) \\
                     & + \dfrac{1-\log(4\pi)}{4\pi}+\gem A_2(t)+B_2(t)
                       + \frac{\C_{\infty,\infty}^{\Ga_0(p^2)}}{2}+\frac{t+1}{2 v_{\GaP}}
                       + A_4(t).
\end{align*}
Moreover writing $\Rpar_0(t) = \dfrac{t+1}{2 v_{\GaP}} + \Epar(t)$ we have
\begin{align*}
    \lim_{t \to \infty} \Epar(t) = & \frac{1}{2 v_{\Ga_0(p)}} \frac{p+1}{p} 
                                     \left(3\gem + \frac{a\pi}{6} - \log (p^2) + 2 B_1(\infty) \right)
                                     + \frac{1-\log(4\pi)}{4\pi} + \frac{\C_{\infty,\infty}^{\GaP}}{2}\\
                                 = & \frac{1-\log(4\pi)}{4\pi}+ \frac{\C_{\infty,\infty}^{\GaP}}{2} + O\left(\frac{\log (p^2)}{p^2} \right).   
\end{align*}
We only need to show that $\C_{\infty, \infty}^{\GaP} = O\left(\frac{\log(p^2)}{p^2}\right)$ 
which follows from our computation of $\phi_{\infty, \infty}^{\GaP}$ in Lemma~\ref{phiinfty}. This concludes the proof of Proposition~\ref{contb}.
\end{proof}

\subsection{Asymptotics of the canonical Green's function}

\begin{proof}[Proof of Theorem ~\ref{RSconst}]
From Equation~\eqref{key-form}, we have
\[
g_{p^2}R_F(s)=-R_D(t,s) +R_{H}(t,s) + R_{E}(t,s) + R_{P-C}(t,s).
\]
Hence, we get the following equality:
\[
g_{p^2}\RG = -\Rdis_0(t) + \Rhyp_0(t) + \Rell_0(t) + \Rpar_0(t). 
\]
According to the Proposition~\ref{contb} as $t \to \infty$, we have $\Rdis_0(t)\to 0, 
\Rell_0=o(\log(p))$. We also have $ \Rhyp_0(t) = \Ehyp(t) - \frac{t-1}{2 v_{\GaP}}$
and $\Rpar_0(t) = \frac{t+1}{2 v_{\GaP}} + \Epar(t)$ and hence 
\begin{align*}
  \Rhyp_0(t) +\Rpar_0(t)=\Ehyp(t)+ \Epar(t)+\frac{1}{ v_{\GaP}}. 
\end{align*}
From Proposition~\ref{contb}, recall that
 $\lim_{t \to \infty} \Ehyp(t) = \frac{1}{2 v_{\GaP}} O_{\epsilon}(p^{2 \epsilon})$
and   $\lim_{t \to \infty} \Epar(t) = 
\frac{1-\log(4\pi)}{4\pi} + O\left(\frac{\log (p^2)}{p^2}\right)$.
As $t \to \infty$, we have
\begin{eqnarray*}
\RG
&= &
\frac{1}{g_{p^2}} \lim_{t \rightarrow \infty} \left(-\Rdis_0(t) + \Rhyp_0(t) + \Rell_0(t) + \Rpar_0(t)\right)\\
&= &
\frac{1}{g_{p^2}} \lim_{t \rightarrow \infty} \left(\Rhyp_0(t) +\Rpar_0(t)+\Rell_0\right) \\
&= &
\frac{1}{g_{p^2}}\lim_{t \rightarrow \infty} \left(\Ehyp(t)+ \Epar(t)+\frac{1}{ v_{\GaP}} \right)+o\left( \frac{\log(p^2)}{g_{p^2}} \right) \\
&= &
\frac{1}{g_{p^2}} \left( \frac{1-\log(4\pi)}{4\pi} + O\left(\frac{\log (p^2)}{p^2}\right)+ \frac{1}{2 v_{\GaP}} O_{\epsilon}(p^{2 \epsilon})+\frac{1}{ v_{\GaP}} \right)+o\left( \frac{\log(p^2)}{g_{p^2}} \right)\\
&= & o\left( \frac{\log(p^2)}{g_{p^2}} \right).\\
\end{eqnarray*}
The last equality follows from $v_{\GaP} = \frac{\pi}{3}p(p+1)$. 
\end{proof}

\begin{proposition} \label{lem:analysis-main}
 For $p > 7$, we have the following asymptotic expression
  \begin{equation*}
    \Gcan(\infty,0) = \frac{6\log (p^2)}{p(p+1)} + o\left(\frac{\log(p^2)}{g_{p^2}}\right).
  \end{equation*}
\end{proposition}

\begin{proof}
Using Theorem~\ref{RSconst}, Remark~\ref{genus-form}  and \eqref{gcanimp}, we obtain
\[
 \Gcan(\infty,0) = -2\pi \CG + o\left(\frac{\log(p^2)}{g_{p^2}}\right).
\]
By Corollary~\ref{Secondterm}, we also have
\begin{align*}
 -2\pi \CG =  \frac{6 \log (p^2)}{p(p+1)} + o\left(\frac{\log(p^2)}{g_{p^2}}\right);
\end{align*}
noting that $v_{\GaP} = \frac{\pi}{3}p(p+1)$. Hence, asymptotically the main contribution for 
$\Gcan(\infty,0)$ comes from $\CG$ and the proposition follows.
\end{proof} 

\begin{remark}
We need the assumption $p > 7$  in the parabolic part as the computations are carried out under the assumption that  $g_{p^2} > 1$.
\end{remark}
In \cite{MR3343899}, the estimates on Arakelov-Green's functions are provided  for general non-compact orbisurfaces. 

\section{Minimal regular models of Edixhoven} \label{sec:MinimalModel}
For primes $p \geq 7$, the modular curve $X_0(p^2)$ is an algebraic curve defined over $\qq$. 

In \cite{MR1056773}, Bas Edixhoven constructed regular models $\tcX_0(p^2)$ for all such primes. Note that the 
above mentioned models are arithmetic surfaces over $\spec \zz$. These models however are not minimal. In 
this section, we recall the regular models of Edixhoven and describe the minimal regular models obtained from 
them.

For any prime $q$ of $\zz$ such that $q \neq p$ the fiber $\tcX_0(p^2)_{\ff_q}$ is a smooth curve of genus 
$g_{p^2}$, the genus of $X_0(p^2)$. For the prime $p$ the fiber $\tcX_0(p^2)_{\ff_p}$ is reducible and non-reduced, 
of arithmetic genus $g_{p^2}$, and whose geometry depends on the class of $p$ in $\zz/ 12\zz$. To describe 
$\tcX_0(p^2)$ it is thus enough to describe the special fiber. 

The minimal regular model $\cX_0(p^2)$ is obtained from $\tcX_0(p^2)$ by three successive blow downs 
of curves in the special fiber $\tcX_0(p^2)_{\ff_p}$ and we shall denote by $\pi: \tcX_0(p^2) \to \cX_0(p^2)$ 
the morphism from Edixhoven's model. 
In the computations, we shall use \cite[Chapter 9, Theorem~2.12]{MR1917232} repeatedly. 
Let $\cdot$ be the local intersection pairing as in \cite{MR1917232}. 

In the following subsections, we shall explicitly describe the special fiber of the minimal regular model 
$\cX_0(p^2)$. We shall also compute the local intersection numbers among the various components in the fiber. 
The Arakelov intersections in this case are obtained by simply multiplying the local intersection numbers 
by $\log(p)$. The following proposition is the key finding of this section:

\begin{proposition} \label{keyprop}
  The special fiber of $ \cX_0(p^2)$ consists of two curves $C_{2,0}'$ and $C_{0,2}'$  and the 
  local intersection numbers are given by:
  \begin{equation*}
    C_{2,0}'\cdot C_{0,2}'= - (C_{2,0}')^2 = - (C_{0,2}')^2 = \frac{p^2-1}{24}.
  \end{equation*} 
\end{proposition}

The proof of the previous proposition shall take up the rest of this section. It is divided into four cases
depending on the residue of $p$ modulo $12$.

\subsection{Case $p \equiv 1 \pmod{12}$.}
Following Edixhoven~\cite{MR1056773}, we draw the special fiber $V_p = \tcX_0(p^2)_{\ff_p}$ in Figure~\ref{fig:1mod12} in this case. Each component is a $\pp^1$ and the pair
$(n,m)$ adjacent to each component denotes the multiplicity of the component $n$ and the local 
self-intersection number $m$. The arithmetic genus is given by $g_{p^2} = 12k^2 -3k -1$ where $p = 12k+1$.
\begin{figure}
  \begin{center}
    \includegraphics[scale=0.8]{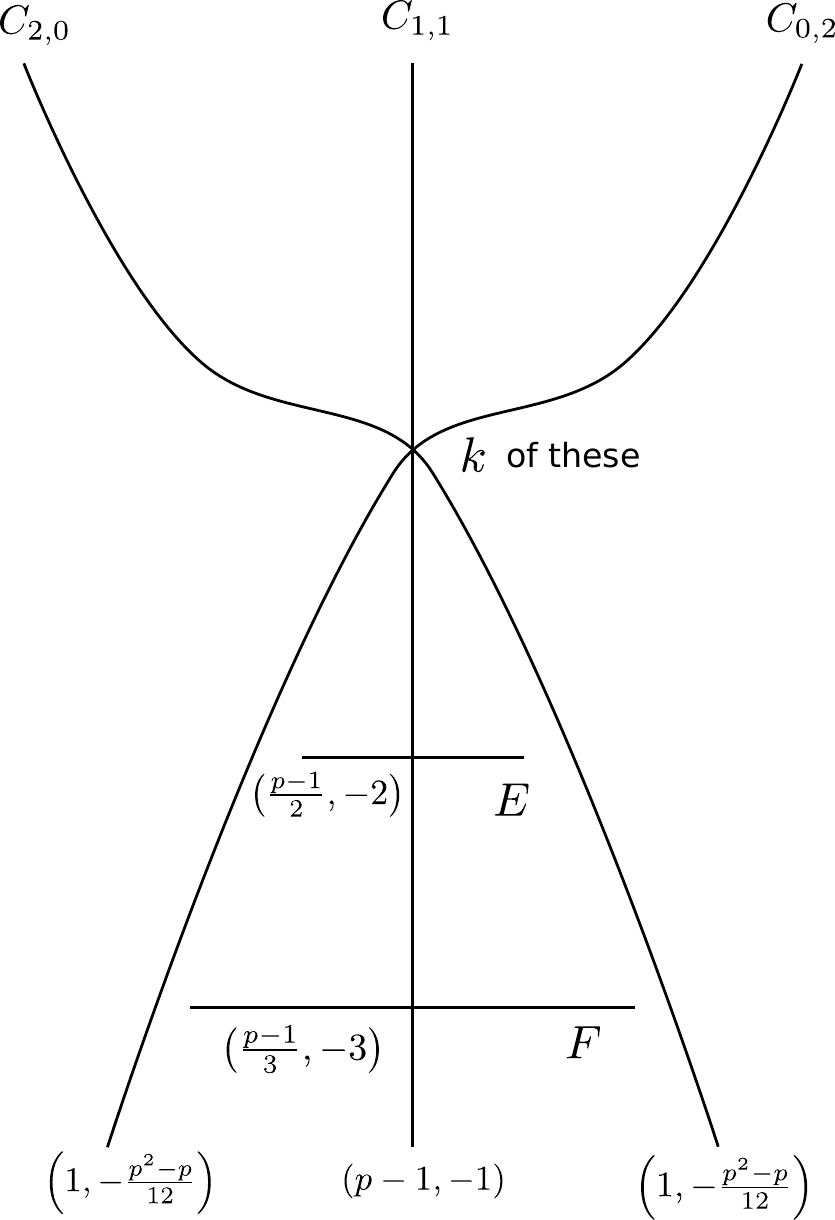}
  \end{center}      
  \caption{The special fiber $\tcX_0(p^2)_{\ff_p}$ when $p \equiv 1 \pmod{12}$.} 
  \label{fig:1mod12} 
\end{figure}

\begin{proposition} \label{pmod1imp}
  The local intersection numbers of the vertical components  supported on the special fiber of $\tcX_0(p^2)$ are given 
  in the following table.
  \begin{equation*}
    \renewcommand*{\arraystretch}{2}
    \begin{array}{l|ccccc}  
      & C_{2,0} & C_{0,2} & \hphantom{0} C_{1,1} \hphantom{0} & \hphantom{000} E \hphantom{000} 
      & \hphantom{000} F \hphantom{000} \\ \hline
      C_{2,0} & -\frac{p(p-1)}{12} & \frac{p-1}{12} & \frac{p-1}{12} & 0 & 0 \\
      C_{0,2} & \frac{p-1}{12} & -\frac{p(p-1)}{12} & \frac{p-1}{12} & 0 & 0 \\
      C_{1,1} & \frac{p-1}{12} & \frac{p-1}{12} & -1 & 1 & 1 \\        
      E & 0 & 0 & 1 & -2 & 0 \\
      F & 0 & 0 & 1 & 0 & -3
    \end{array}
  \end{equation*}
\end{proposition}

\begin{proof}
  The self-intersections $C_{1,1}^2, E^2$ and $F^2$ were calculated by Edixhoven (see \cite[Fig. 1.5.2.1]{MR1056773}). 
  Since $V_p$ is the principal divisor $(p)$, we must have $V_p \cdot D = 0$ for any vertical divisor $D$. Moreover, 
  $V_p = C_{2,0} + C_{0,2} + (p-1) C_{1,1} + \frac{p-1}{2} E  + \frac{p-1}{3} F$ is the linear combination of all 
  the prime divisors of the special fiber counted with multiplicities. All the other intersection numbers can be 
  easily calculated using these information. For example $V_p\cdot E = 0$ gives 
  $(p-1) C_{1,1}\cdot E + \frac{(p-1)}{2} E^2 = 0$, now since $E^2 = -2$ we have $C_{1,1}\cdot E = 1$. The other
  calculations are analogous and we omit them here.
\end{proof}

\begin{proof}[Proof of Proposition~\ref{keyprop} when $p \equiv 1 \pmod {12}$]
  By Proposition~\ref{pmod1imp}, note that the component $C_{1,1}$ is rational and has 
  self-intersection $-1$. By Castelnuovo's criterion \cite[Chapter~9, Theorem~3.8]{MR1917232} we can thus 
  blow down $C_{1,1}$ without introducing a singularity. Let $\cX_0(p^2)'$ be the corresponding arithmetic surface 
  and $\pi_1: \tcX_0(p^2) \to \cX_0(p^2)'$, be the blow down morphism.

  For $E'  = \pi_1(E)$,  we see that $\pi_1^* E' = E + C_{1,1}$. In fact $\pi_1^* E' = E + \mu C_{1,1}$ 
  (see Liu \cite[Chapter 9, Proposition 2.18]{MR1917232}). 
  Using  \cite[Chapter 9, Theorem~2.12]{MR1917232}, we obtain $0 = C_{1,1}\cdot \pi_1^* E' = 1  - \mu$. 
  Hence, we deduce that  $(E')^2 = (\pi_1^* E')^2 = \langle E + C_{1,1}, E +C_{1,1} \rangle = -1$. Thus $E'$  
  is a rational curve in the special fiber of $\cX_0(p^2)'$ with self intersection $-1$. It can thus be blown 
  down again and the resulting scheme is again regular. Let $\cX_0(p^2)''$ be the blow down and $\pi_2: \tcX_0(p^2) 
  \to \cX_0(p^2)''$ the morphism from $\tcX_0(p^2)$. 

  Let $F' = \pi_2(F)$, and if $\pi_2^*F' = F + \mu C_{1,1} + \nu E$ for $\mu, \nu \in \qq$ then using the 
  fact that $C_{1,1} \cdot \pi_2^*F' = E \cdot \pi_2^*F' = 0$ we find $\mu = 2$ and $\nu =1$. This 
  yields $\pi_2^*F' = F + 2 C_{1,1} + E$ and hence $(F')^2 = -1$.  We can thus blow down $F'$  further to 
  arrive finally at an arithmetic surface $\cX_0(p^2)$. This is the minimal regular model of $X_0(p^2)$ since 
  no further blow down is possible. Let $\pi: \tcX_0(p^2) \to \cX_0(p^2)$ be the morphism obtained by composing
  the sequence of blow downs.

  The special fiber of $\cX_0(p^2)$ consists of two curves $C_{2,0}'$ and $C_{0,2}'$, that are the images 
  of $C_{2,0}$ and $C_{0,2}$ respectively under $\pi$. They intersect with high multiplicity at a single point. 
  To calculate the intersections we notice that 
  \begin{align*}
    \pi^* C_{2,0}' = C_{2,0} + \frac{p-1}{2} C_{1,1} + \frac{p-1}{4} E + \frac{p-1}{6} F, \\[10pt]
    \pi^* C_{0,2}' = C_{0,2} + \frac{p-1}{2} C_{1,1} + \frac{p-1}{4} E + \frac{p-1}{6} F,
  \end{align*} 
  obtained as before from the fact that the intersections of $\pi^* C'_{2,0}$ and $\pi^* C'_{0,2}$ 
  with $C_{1,1}$, $E$ and $F$ are all $0$. This yields
  \begin{equation*}
    C_{2,0}'\cdot C_{0,2}'= - (C_{2,0}')^2 = - (C_{0,2}')^2 = \frac{p^2-1}{24}.
  \end{equation*}
\end{proof}

\subsection{Case $p \equiv 5 \pmod{12}$.}

In this case the special fiber $V_p = \tcX_0(p^2)_{\ff_p}$ is described by 
Figure~\ref{fig:5mod12}. Each component is a $\pp^1$ and the genus is given by 
$g_{p^2} = 12k^2 + 5k$ where $p = 12k+5$.

\begin{figure}[h]
  \begin{center}
    \includegraphics[scale=0.8]{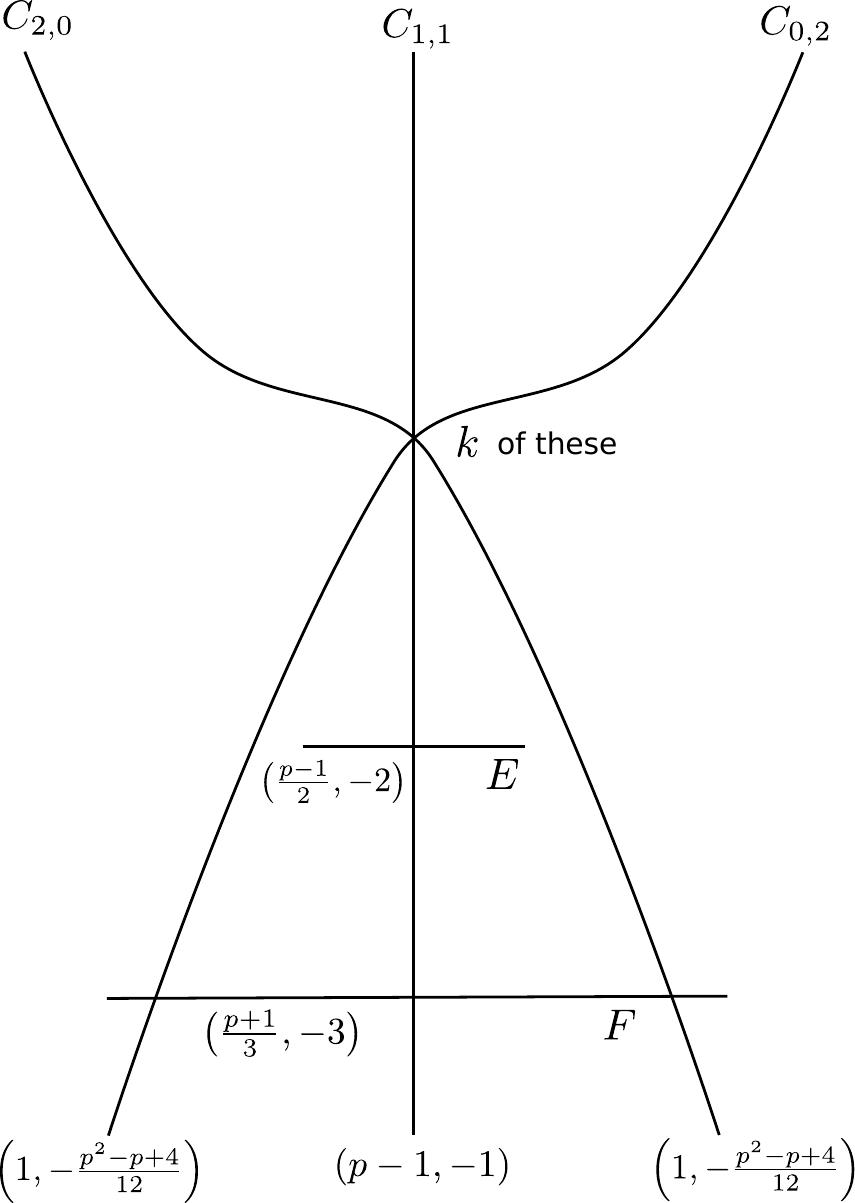}
  \end{center}
  \caption{The special fiber $\tcX_0(p^2)_{\ff_p}$ when $p \equiv 5 \pmod{12}$.}
  \label{fig:5mod12}
\end{figure}  

\begin{proposition} \label{pmod5imp}
  The local intersection numbers of the vertical components  supported on the special fiber of 
  $\tcX_0(p^2)$ for $p \equiv 5 \pmod {12}$ are given in the following table.  
  \begin{equation*} \displaystyle
    \renewcommand*{\arraystretch}{2}
    \begin{array}{l|ccccc}
       & C_{2,0} & C_{0,2} & \hphantom{0} C_{1,1} \hphantom{0} & \hphantom{000} E \hphantom{000} 
       & \hphantom{000} F \hphantom{000} \\ \hline
    C_{2,0} & -\frac{p^2-p+4}{12} & \frac{p-5}{12} & \frac{p-5}{12} & 0 & 1 \\
    C_{0,2} & \frac{p-5}{12} & -\frac{p^2-p+4}{12} & \frac{p-5}{12} & 0 & 1 \\
    C_{1,1} & \frac{p-5}{12} & \frac{p-5}{12} & -1 & 1 & 1 \\        
    E & 0 & 0 & 1 & -2 & 0 \\
    F & 1 & 1 & 1 & 0 & -3
    \end{array}
  \end{equation*}
\end{proposition} 

The calculations are very similar to the previous case hence we omit the proof.

\begin{proof}[Proof of Proposition~\ref{keyprop} when $p \equiv 5 \pmod {12}$]
  The minimal regular model is obtained by blowing down $C_{1,1}$, then the image of $E$ and then the 
  image of $F$ as in the previous section. We again denote the minimal regular model by $\cX_0(p^2)$,
  and by $\pi: \tcX_0(p^2) \to \cX_0(p^2)$ the morphism obtained by the successive blow downs.

  The special fiber of $\cX_0(p^2)$ again consists of two curves $C_{2,0}'$ and $C_{0,2}'$ that are the 
  images of $C_{2,0}$ and $C_{0,2}$ respectively under $\pi$. The curves $C_{2,0}'$ and $C_{0,2}'$ 
  intersect at a single point. In this case we have
  \begin{align*}
    \pi^* C_{2,0}' = C_{2,0} + \frac{p-1}{2} C_{1,1} + \frac{p-1}{4} E + \frac{p+1}{6} F, \\[10pt]
    \pi^* C_{0,2}' = C_{0,2} + \frac{p-1}{2} C_{1,1} + \frac{p-1}{4} E + \frac{p+1}{6} F,
  \end{align*} 
  obtained as before from the fact that the intersections of $\pi^* C'_{2,0}$ and $\pi^* C'_{0,2}$ 
  with $C_{1,1}$, $E$ and $F$ are all $0$. This yields
  \begin{equation*}
    C_{2,0}' \cdot C_{0,2}' = - (C_{2,0}')^2 = - (C_{0,2}')^2 = \frac{p^2-1}{24}.
  \end{equation*}
\end{proof}

\subsection{Case $p \equiv 7 \pmod{12}$.}

In this case the special fiber $V_p = \tcX_0(p^2)_{\ff_p}$ is described by Figure~\ref{fig:7mod12}. 
Each component is a $\pp^1$ occurring with the specified multiplicity. The genus is given by 
$g_{p^2} = 12k^2 + 9k+1$ where $p = 12k+7$. 
\begin{figure}[h]
  \begin{center}
    \includegraphics[scale=0.8]{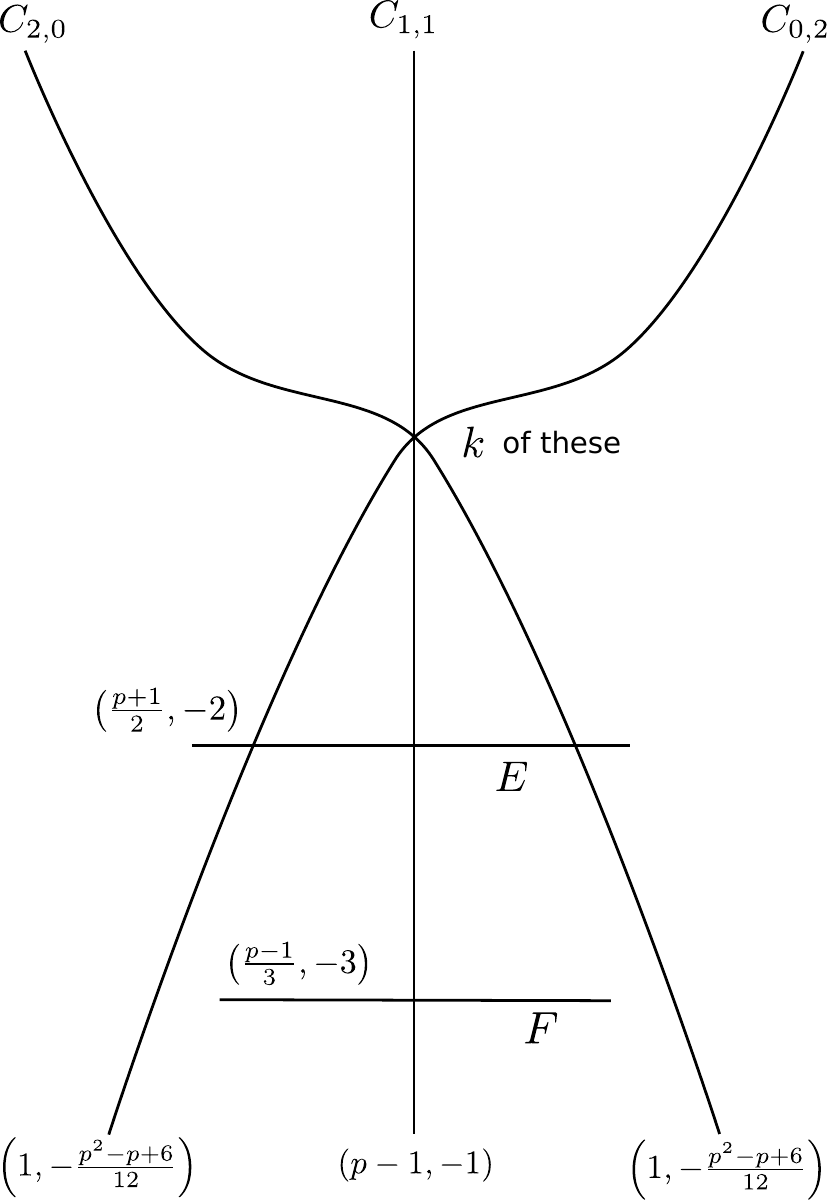}
  \end{center}
  \caption{The special fiber $\tcX_0(p^2)_{\ff_p}$ when $p \equiv 7 \pmod{12}$.}
  \label{fig:7mod12}
\end{figure}

\begin{proposition} \label{pmod7imp}
  The local intersection numbers of the prime divisors supported on the special fiber
  of $\tcX_0(p^2)$ for $p \equiv 7 \pmod {12}$ are given in the following table.
  \begin{equation*}
    \renewcommand*{\arraystretch}{2}
    \begin{array}{l|ccccc}
       & C_{2,0} & C_{0,2} & \hphantom{0} C_{1,1} \hphantom{0} & \hphantom{000} E \hphantom{000} 
       & \hphantom{000} F \hphantom{000} \\ \hline
    C_{2,0} & -\frac{p^2-p+6}{12} & \frac{p-7}{12} & \frac{p-7}{12} & 1 & 0 \\
    C_{0,2} & \frac{p-7}{12} & -\frac{p^2-p+6}{12} & \frac{p-7}{12} & 1 & 0 \\
    C_{1,1} & \frac{p-7}{12} & \frac{p-7}{12} & -1 & 1 & 1 \\
    E & 1 & 1 & 1 & -2 & 0 \\
    F & 0 & 0 & 1 & 0 & -3
    \end{array}
  \end{equation*}
\end{proposition} 

\begin{proof}[Proof of Proposition~\ref{keyprop} when $p \equiv 7 \pmod {12}$]
  The minimal regular model is obtained by blowing down $C_{1,1}$, then the image of $E$ and then the 
  image of $F$ as in the previous sub-section. Let $\cX_0(p^2)$ be the minimal regular model
  and $\pi: \tcX_0(p^2) \to \cX_0(p^2)$ the morphism obtained by the successive blow downs.

  The special fiber of $\cX_0(p^2)$ consists of two curves $C_{2,0}'$ and $C_{0,2}'$ that are the images of 
  $C_{2,0}$ and $C_{0,2}$ respectively under $\pi$. The curves $C_{2,0}'$ and $C_{0,2}'$ intersect at a single 
  point. Here
  \begin{align*}
    \pi^* C_{2,0}' = C_{2,0} + \frac{p-1}{2} C_{1,1} + \frac{p+1}{4} E + \frac{p-1}{6} F, \\[10pt]
    \pi^* C_{0,2}' = C_{0,2} + \frac{p-1}{2} C_{1,1} + \frac{p+1}{4} E + \frac{p-1}{6} F,
  \end{align*} 
  easily calculated using fact that the intersections of $\pi^* C'_{2,0}$ and $\pi^* C'_{0,2}$ 
  with $C_{1,1}$, $E$ and $F$ are all $0$. This yields
  \begin{equation*}
    C_{2,0}' \cdot C_{0,2}'  = - (C_{2,0}')^2 = - (C_{0,2}')^2 = \frac{p^2-1}{24}.
  \end{equation*}
\end{proof}

\subsection{Case $p \equiv 11 \pmod{12}$.}

In this final case the special fiber $V_p = \tcX_0(p^2)_{\ff_p}$ is described by Figure~\ref{fig:11mod12}. 
Each component is a $\pp^1$. The genus is given by $g_{p^2} = 12k^2 + 17k+6$ where $p = 12k+11$.

\begin{figure}[h]
  \begin{center}
    \includegraphics[scale=0.8]{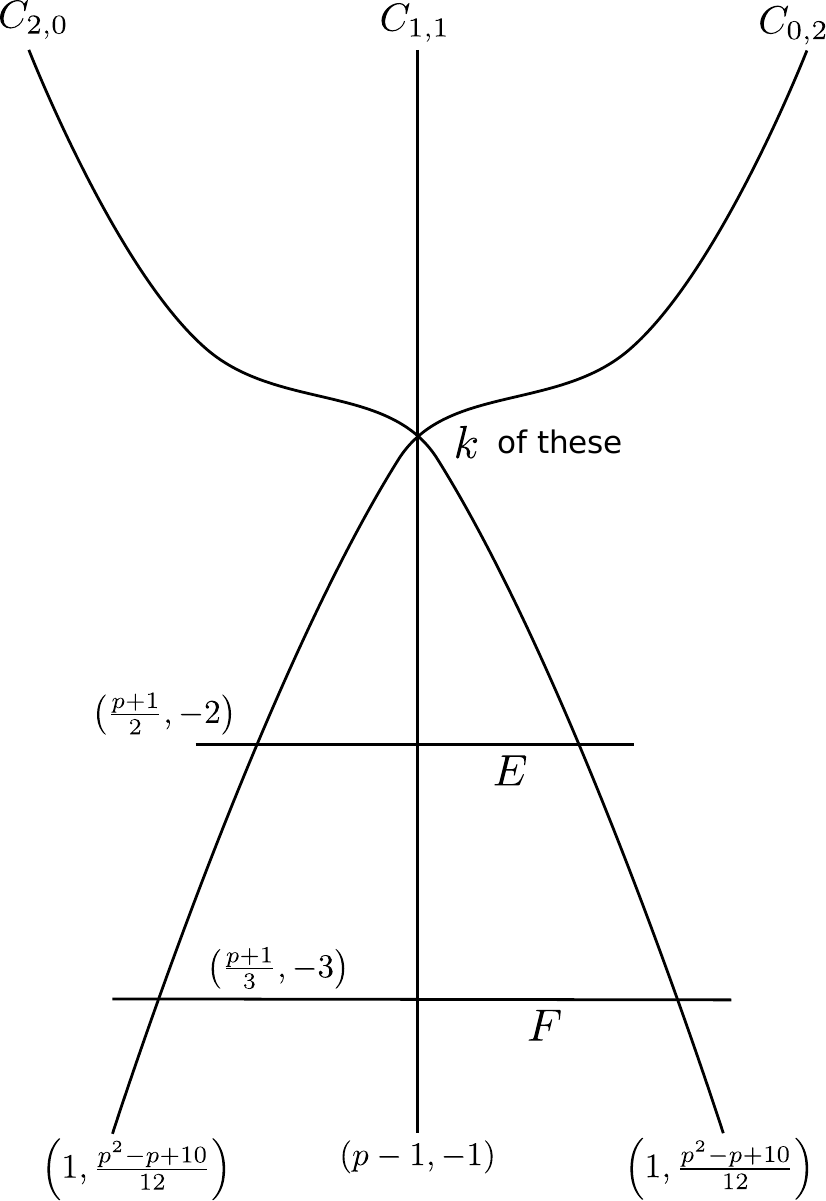}
  \end{center}
  \caption{The special fiber $\tcX_0(p^2)_{\ff_p}$ when $p \equiv 11 \pmod{12}$.}
  \label{fig:11mod12}
\end{figure}

\begin{proposition} \label{pmod11imp}
  The local intersection numbers of the prime divisors supported on the special fiber
  of $\tcX_0(p^2)$ for $p \equiv 11 \pmod {12}$ are given in the following table.
  \begin{equation*}
    \renewcommand*{\arraystretch}{2}
    \begin{array}{l|ccccc}
       & C_{2,0} & C_{0,2} & \hphantom{00} C_{1,1} \hphantom{00} & \hphantom{000} E \hphantom{000} 
       & \hphantom{000} F \hphantom{000} \\ \hline
    C_{2,0} & -\frac{p^2-p+10}{12} & \frac{p-11}{12} & \frac{p-11}{12} & 1 & 1 \\
    C_{0,2} & \frac{p-11}{12} & -\frac{p^2-p+10}{12} & \frac{p-11}{12} & 1 & 1 \\
    C_{1,1} & \frac{p-11}{12} & \frac{p-11}{12} & -1 & 1 & 1 \\
    E & 1 & 1 & 1 & -2 & 0 \\
    F & 1 & 1 & 1 & 0 & -3
    \end{array}
  \end{equation*}
\end{proposition}

Let us now complete the proof of Proposition~\ref{keyprop} by presenting the final case. 

\begin{proof}[Proof of Proposition~\ref{keyprop} when $p \equiv 11 \pmod {12}$]
  The minimal regular model is again obtained by blowing down $C_{1,1}$, then the image of $E$ and then the 
  image of $F$ as in the previous section. We again denote the minimal regular model by $\cX_0(p^2)$.
  Let $\pi: \tcX_0(p^2) \to \cX_0(p^2)$ be the morphism obtained by the successive blow downs.

  The special fiber of $\cX_0(p^2)$ consists of two curves $C_{2,0}'$ and $C_{0,2}'$ that are the images of 
  $C_{2,0}$ and $C_{0,2}$ respectively under $\pi$ intersecting at a single point. Mimicking the calculations 
  of the previous sub-section
  \begin{align*}
    \pi^* C_{2,0}' = C_{2,0} + \frac{p-1}{2} C_{1,1} + \frac{p+1}{4} E + \frac{p+1}{6} F, \\[10pt]
    \pi^* C_{0,2}' = C_{0,2} + \frac{p-1}{2} C_{1,1} + \frac{p+1}{4} E + \frac{p+1}{6} F.
  \end{align*} 
  This again yields
  \begin{equation*}
    C_{2,0}' \cdot C_{0,2}' = - (C_{2,0}')^2 = - (C_{0,2}')^2 = \frac{p^2-1}{24}.
  \end{equation*}
\end{proof}

\section{Algebraic part of self-intersection}
We continue with the notation from Section \ref{sec:MinimalModel}. Let $H_0$ and $H_{\infty}$ be the 
sections of $\cX_0(p^2)/\zz$ corresponding to the cusps $0, \infty \in X_0(p^2)(\qq)$. The horizontal divisor 
$H_0$ intersects exactly one of the curves of the special fiber at an $\ff_p$ rational point transversally 
(cf.  Liu~\cite[Chapter~9, Proposition~1.30 and Corollary~1.32]{MR1917232}). We call that component
$C_0'$. It follows from the cusp and component labelling of Katz and Mazur~\cite[p.~296]{MR772569} that 
$H_{\infty}$ meets the other component transversally and we call it $C_{\infty}'$. 
The components $C_0'$ and $C_{\infty}'$ intersect in a single point. 

Recall that the local intersection numbers are given by [cf. Proposition \ref{keyprop}]:
\begin{equation} \label{eq:intersectionnumber}
  C_0'\cdot C_{\infty}' = - (C_0')^2 = -(C_{\infty}')^2 = \frac{p^2 - 1}{24}.
\end{equation}
Let $K_{\cX_0(p^2)}$ be a canonical divisor of $\cX_0(p^2)$, that is any divisor whose corresponding line
bundle is the relative dualizing sheaf. We then have the following result:

\begin{lemma} \label{lem:divisors}
  For $s_p = \dfrac{p^2-1}{24}$, consider the vertical divisors $V_0 = -\dfrac{g_{p^2} - 1}{s_p} C_0'$ and 
  $V_{\infty} = -\dfrac{g_{p^2} - 1}{s_p} C_{\infty}'$.  The divisors 
  \begin{equation*}
    D_m = K_{\cX_0(p^2)} - (2g_{p^2} -2)H_m + V_m, \qquad m\in \{0, \infty\}
  \end{equation*}
  are orthogonal to all vertical divisors of $\cX_0(p^2)$ with respect to the Arakelov intersection
  pairing.
\end{lemma}

\begin{proof}
  For any prime $q \neq p$ if $V$ is the corresponding fiber over $(q) \in \spec \zz$, 
  then $\langle V_m, V \rangle = 0$. Moreover, by the adjunction formula \cite[Chapter 9, Proposition 1.35]
  {MR1917232}, $\langle K_{\cX_0(p^2)}, V \rangle = (2 g_{p^2} - 2)\log p$. 
  The horizontal divisor $H_m$ meets any fiber transversally at a smooth $\ff_p$ rational point which gives 
  $\langle D_m, V \rangle = 0$. 
  
  Again using the adjunction formula
  \begin{equation*}
    \langle K_{\cX_0(p^2)}, C_0' + C_{\infty}' \rangle = (2 g_{p^2} - 2)\log p,
  \end{equation*}
  and on the other hand from the discussion of $\cX_0(p^2)$ it is clear that 
  $\langle K_{\cX_0(p^2)}, C_0' \rangle = \langle K_{\cX_0(p^2)}, C_{\infty}' \rangle$,
  hence we have 
  \begin{equation*}
    \langle K_{\cX_0(p^2)}, C_0' \rangle = \langle K_{\cX_0(p^2)}, C_{\infty}' \rangle = 
    (g_{p^2} - 1)\log p.
  \end{equation*}
    
  If $m \in \{0,\infty\}$, then
  \begin{equation*}
     \langle D_m, C_m' \rangle = (g_{p^2} - 1)\log p - (2g_{p^2} - 2)\log p + (g_{p^2} - 1)\log p = 0.
   \end{equation*} 
   
   Finally if $n \in \{0, \infty\}$ and $n \neq m$ then 
   \begin{equation*}
     \langle D_m, C_n' \rangle = (g_{p^2} - 1)\log p - 0 - (g_{p^2} - 1)\log p = 0.
   \end{equation*}
   This completes the proof.
\end{proof} 

The proof of the following lemma is analogous to that of Proposition D of Abbes-Ullmo~\cite{MR1437298}.

\begin{lemma} 
 \label{lem:selfinter}
  For $m \in \{0, \infty\}$, consider the horizontal divisors $H_m$ as above.  We have the following equality of 
  the Arakelov self-intersection number of the relative dualizing sheaf:
   \begin{equation*}
    (\overline{\omega}_{p^2})^2 = -4g_{p^2}(g_{p^2}-1) \langle H_0, H_{\infty}\rangle
                                         + \dfrac{(g_{p^2}^2 - 1)\log p}{s_p}+e_p
  \end{equation*}
  with 
  \begin{equation*}
    e_p=
    \begin{cases}
    0 & \text{if $p \equiv 11 \pmod {12}$,} \\
    O(\log p)&  \text{if $p \not \equiv 11 \pmod {12}$. } \\
    \end{cases}
  \end{equation*}
\end{lemma}

\begin{proof}
  Since $D_m$ has degree 0 and is perpendicular to vertical divisors, 
  by a theorem of Faltings-Hriljac~\cite[Theorem 4]{MR740897}, we obtain:
  \begin{equation} \label{eq:ineq2}
    \langle D_m, D_m \rangle = -2 \Big( \text{N\'{e}ron-Tate height of } \cO(D_m) \Big)
    :=h_m.
  \end{equation}
  This yields
  \begin{equation*}
    \langle D_m, \  K_{\cX_0(p^2)} - (2g_{p^2}-2)H_m  \rangle = h_m.
  \end{equation*}
  The previous expression expands to 
  \begin{equation*}
   \overline{\omega}_{p^2}^2 = 
    -(2g_{p^2}-2)^2 H_m^2 + 2(2g_{p^2}-2) \langle K_{\cX_0(p^2)}, H_m\rangle - 
               \langle K_{\cX_0(p^2)}, V_m \rangle +
               (2g_{p^2}-2) \langle H_m, V_m\rangle+h_m.
  \end{equation*}
   
  Now using the equality $\langle D_m, V_m \rangle = 0$ which yields $\langle K_{\cX_0(p^2)}, V_m \rangle
  - (2g_{p^2}-2) \langle H_m, V_m \rangle + V_m^2 = 0$ and the adjunction formula $\langle K_{\cX_0(p^2)}, H_m 
  \rangle = - H_m^2$, (see Lang \cite[Ch. IV, Sec. 5, Corollary 5.6]{MR969124}), we get:
  \begin{equation*}
    (\overline{\omega}_{p^2})^2 = -4g_{p^2}(g_{p^2}-1)H_m^2 + V_m^2+h_m.
  \end{equation*}
  This yields 
  \begin{equation} \label{eq:ineq3}
    (\overline{\omega}_{p^2})^2 = -2g_{p^2}(g_{p^2}-1) (H_0^2 + H_{\infty}^2) 
    + \frac{1}{2}(V_0^2 + V_{\infty}^2)+ \frac{1}{2}(h_0 +h_{\infty}).
  \end{equation}
    
  Consider the divisor $D_{\infty} - D_0 = (2g_{p^2}-2)(H_0 - H_{\infty}) + (V_{\infty} - V_0)$. 
  The generic fiber of the the line bundle corresponding to the above divisor is supported on cusps. Hence by the  
  Manin-Drinfeld theorem \cite{MR0314846, MR0318157}, $D_{\infty}- D_0$ is a torsion element of the Jacobian
  $J_0(p^2)$. Moreover the divisor $D_{\infty}- D_0$ satisfies the hypothesis of the Faltings-Hriljac theorem, 
  which along with the vanishing of Neron-Tate height at torsion points implies 
  $ \langle D_0- D_{\infty},  D_0- D_{\infty}\rangle=0$. Hence, we obtain
  \begin{equation*}
    H_0^2 + H_{\infty}^2 = 2 \langle H_0, H_{\infty} \rangle + 
                           \frac{V_0^2 - 2 \langle V_0, V_{\infty}\rangle + V_{\infty}^2}{(2g_{p^2}-2)^2}.
  \end{equation*}
  
  Substituting this in \eqref{eq:ineq3} we deduce
  \begin{equation*}
    (\overline{\omega}_{p^2})^2 
    = -4g_{p^2}(g_{p^2}-1) \langle H_0, H_{\infty}\rangle - \frac{1}{2g_{p^2}-2} (V_0^2 + V_{\infty}^2) + 
               \frac{g_{p^2}}{g_{p^2}-1} \langle V_0, V_{\infty}\rangle+ \frac{1}{2}(h_0 +h_{\infty}).
  \end{equation*}
  
  For  $p \equiv 11 \pmod {12}$, the modular curve $X_0(p^2)$ has no elliptic points. We deduce that for 
  $m \in \{0, \infty\}$, the divisors $D_m$ are  supported at cusps and hence $h_0=h_{\infty}=0$ (see 
  \cite[Lemma 4.1.1]{MR1437298}). 
 
  If $p \not\equiv 11 \pmod {12}$, the canonical divisor $K_{\cX_0(p^2)}$ is supported at the set of cusps 
  and the set of elliptic points. By Manin-Drinfeld theorem, the N\'{e}ron-Tate heights $h_{NT}$ of the 
  divisors supported at cusps is zero.  We provide a bound on the N\'{e}ron-Tate heights of elliptic points 
  by a computation similar to \cite[Section 6, equation 36]{MR1614563}.  
  Let $f_{p^2}: X_0(p^2) \rightarrow X_0(1)=\pp^1$ be the natural projection. Let $i$ and $j$ be the points 
  on $X_0(1)$ corresponding to to the points $i$ and $j=e^{\frac{2 \pi \sqrt{-1}}{3}}$ of the complex upper 
  half plane $\hh$. Let $H_i$ (respectively $H_j$) be the divisor of $X_0(p^2)$ consisting of elliptic points
  lying above $i$ (respectively  $j$).
  
  By an application of Hurwitz formula \cite[cf. proof of Lemma 6.1, p. 670]{MR1614563}, we have:
  \[
    K_{\cX_0(p^2)} \sim C 
                        -\frac{1}{2} \sum_{\substack{f_{p^2}(Q_i)=i,\\ e_{Q_i} = 1}} Q_i
                        -\frac{2}{3} \sum_{\substack{f_{p^2}(Q_j)=j,\\ e_{Q_j} = 1}} Q_j , 
  \]
  where $C$ is a divisor with rational coefficients supported at the cusps and $Q_i$ 
  (respectively $Q_j$) are points on $X_0(p^2)$ above $i$ (respectively $j$) with ramification index 
  $e_{Q_i}$ (respectively $e_{Q_j}$).
  
  Hence by an application of the Manin-Drinfeld theorem, we have an equality 
  \cite[Lemma 6.1]{MR1614563}: 
  \begin{equation} \label{eq:heegner}
    h_0 = h_{\infty} = \frac{1}{36} h_{NT}( 3(H_i-v_2 \infty)+4(H_j-v_3 \infty));
  \end{equation}
  with $v_2=(1+(\frac{-1}{p}))$ (the number of elliptic points in $X_0(p^2)$ lying above $i$) 
  and $v_3=(1+(\frac{-3}{p}))$ (the number of elliptic points in $X_0(p^2)$ lying above $j$ ). 
  
  In \cite[Lemma 6.2]{MR1614563} the authors show that the preimages of $i$ under $f_{p^2}$ with ramification
  index $1$ are Heegner points of discriminant $-4$, these are precisely the elliptic points of $X_0(p^2)$ lying
  over $i$, and the preimages of $j$ with ramification index $1$ are Heegner points of index $-3$, these are the 
  elliptic points over $j$.
  
  Let $c$ be an elliptic point of $X_0(p^2)$ lying above $i$ or $j$. 
  By  a formula \cite[p. 307]{MR833192} adapted to our particular modular curve $X_0(p^2)$, we have: 
  \[
    h_{NT}((c)-(\infty))=<c,c>_{\infty}+<c, c>_{\mathrm{fin}},
  \]
  with $<c, c>_{\mathrm{fin}}=2 \log(p^2)$ if $c$ lies above $i$ (respectively,  $<c, c>_{\mathrm{fin}}
  =3 \log(p^2)$ if $c$ lies 
  above $j$) and $<c,c>_{\infty}=O_{\epsilon}(p^{2\epsilon-2})$ \cite[Section 6, p. 671-673]{MR1614563}. 
  By parallelogram law of N\'{e}ron-Tate heights (N\'{e}ron-Tate heights are always positive) and
  using \eqref{eq:heegner} we obtain (see \cite[Section 6, p. 673]{MR1614563}):
  \[
    e_p = \frac{1}{2}(h_0 +h_{\infty}) = O(\log p).
  \]
  
  The lemma now  follows from  \eqref{eq:intersectionnumber}.
\end{proof}

Using the above results, we obtain the Theorem~\ref{MaintheoremDDC} of the paper. 

\begin{proof}[Proof of Theorem~\ref{MaintheoremDDC}]
  By the previous lemma, we have
  \begin{align*}
    (\overline{\omega}_{\cX_0(p^2)})^2 &= -4g_{p^2}(g_{p^2}-1) \langle H_0, H_{\infty}\rangle
      + \dfrac{(g_{p^2}^2 - 1)\log p}{s_p} + e_p\\
    & = 4g_{p^2}(g_{p^2}-1)\Gcan(\infty,0) + \dfrac{(g_{p^2}^2 - 1)\log p}{s_p} + e_p.
  \end{align*}                                          
  An explicit computation of genus [cf. Remark~\ref{genuscomp}] shows that:
  \begin{equation*}
    g_{p^2}-1=\frac{(p+1)(p-6)-12c}{12}
  \end{equation*}
  with $c \in \{0,\frac{1}{2},\frac{2}{3}, \frac{7}{6}\}$. 
  
   By  Proposition~\ref{lem:analysis-main}, we obtain:                                 
  \begin{equation}\label{key}
    4g_{p^2}(g_{p^2}-1)\Gcan(\infty,0)=4g_{p^2} \log p+o(g_{p^2} \log p).
  \end{equation}  
                       
 Furthermore, we deduce the following asymptotic:
  \begin{align*}
    \dfrac{(g_{p^2}^2 - 1)\log p}{s_p} & = \dfrac{(g_{p^2}+1)(g_{p^2}-1) \log p}{s_p}\\
    & = (g_{p^2}+1)\log p[2+o(1)]\\
      &= 2 g_{p^2}\log p+o(g_{p^2}\log p).
  \end{align*}
  Since $e_p=o(g_{p^2}\log p)$, we obtain the theorem.                         
\end{proof}

\bibliographystyle{crelle}
\bibliography{Eisensteinquestion.bib}
\end{document}